\documentclass[10pt]{article}
\usepackage{amsmath,amsthm,amssymb,a4wide} 
\usepackage[margin=2.95cm]{geometry}
\usepackage{graphicx}
\usepackage{epsfig} 
\usepackage{slashed}  
\usepackage[hidelinks]{hyperref}
\usepackage[T1]{fontenc}
\let\OLDthebibliography\thebibliography
\renewcommand\thebibliography[1]{
  \OLDthebibliography{#1}
  \setlength{\parskip}{2pt}
  \setlength{\itemsep}{2pt plus 0.3ex}
}

\usepackage{cite}
\allowdisplaybreaks

\theoremstyle{plain}  
\newtheorem{theorem}{Theorem}[section]
\newtheorem{proposition}[theorem]{Proposition}
\newtheorem{lemma}[theorem]{Lemma}
\newtheorem{corollary}[theorem]{Corollary}

\theoremstyle{remark}

\numberwithin{equation}{section} 
\numberwithin{figure}{section}  
\usepackage{amssymb, amscd, mathrsfs, wasysym} 

\newcommand{\di}{\textrm{d}}
\newcommand{\cf}{\mathcal{C}_f}
\newcommand{\ci}{\mathcal{C}_i}

\newcommand \la \langle
\newcommand \ra \rangle

\newcommand \Kcal {\mathcal{K}}
\newcommand \Hcal {\mathcal{H}}

\newcommand \underdel {\underline \partial}

\newcommand \trianglerightNEW \triangleright

\newcommand \auth {\textsc}

\newcommand \bei {\begin{itemize}}
\newcommand \eei {\end{itemize}}
\newcommand \be {\begin{equation}}
\newcommand \bel {\be\label}
\newcommand \ee {\end{equation}}
\newcommand \bea {\be\aligned}
\newcommand \eea {\endaligned\ee}
\newcommand \del \partial
\newcommand \RR {\mathbb R}

\newcommand \eps \epsilon

\newcommand{\tw}{\tilde{w}}
\newcommand{\tv}{\tilde{v}}
\newcommand{\tW}{\widetilde{W}}
\newcommand{\tWz}{\widetilde{W}_0}
\newcommand{\tV}{\widetilde{V}}
\newcommand{\hV}{\widehat{V}}
\newcommand{\define}{:=}

\usepackage[most]{tcolorbox} 

\let\oldmarginpar\marginpar
\renewcommand\marginpar[1]{\-\oldmarginpar[\raggedleft\footnotesize #1]%
{\raggedright\footnotesize #1}}

\newcommand{\myfootnote}[1]{
    \renewcommand{\thefootnote}{}
    \footnotetext{\hspace{-16.5pt}\scriptsize#1}
    \renewcommand{\thefootnote}{\arabic{footnote}}
}

\begin{document}
\title{\bf \Large 
Two dimensional wave--Klein-Gordon equations \\ with semilinear nonlinearities} 
\author{Shijie Dong${}^\ast$
and Zoe Wyatt${}^\dagger$
}
\date{}
\maketitle

\begin{abstract}
In this paper we investigate the small data global existence and pointwise decay of solutions to two systems of coupled wave--Klein-Gordon equations in two spatial dimensions. In particular, we consider  critical (in the sense of time decay) semilinear nonlinearities for the wave equation and below-critical semilinear nonlinearities for the Klein-Gordon equation, a situation that has not been studied before in the context of coupled wave and Klein-Gordon equations. An interesting feature of our two systems is that the below-critical nonlinearity causes the Klein-Gordon field to lose its linear behaviour close to the light cone, even though it enjoys optimal time decay. 
\end{abstract}

\section{Introduction}
\myfootnote{${}^\ast$Southern University of Science and Technology, SUSTech International Center for Mathematics, and Department of Mathematics, 518055 Shenzhen,  China.
${}^\dagger$Department of Mathematics, King’s College London, Strand Building, Strand, London, WC2R 2LS. \\
{\sl Email addresses}: ${}^\ast$shijiedong1991@hotmail.com, dongsj@sustech.edu.cn, ${}^\dagger$zoe.wyatt@kcl.ac.uk \\
{\sl MSC codes:} 35L05, 35L52, 35L70 \\
}
We are interested in this paper in systems of PDEs of the form
\bea\label{eq:intro11}
- \Box w = F_w, \quad 
- \Box v + v = F_v,
\eea
where the flat wave operator is $-\Box \define \del_t^2 - \sum_{a=1}^n \del_{x^a}^2$. 
The inhomogeneities $F_w, F_v$ may, or may not, couple the two equations together. We consider only spatial dimensions $n\geq 2$. 

The study of systems of the type \eqref{eq:intro11} has a long and active history in the analysis of PDEs and can play an illuminating role in our understanding of mathematical models of physical phenomena. A stimulating goal is to understand those nonlinearities $F_w, F_v$ for which small-data perturbations of the zero solution to \eqref{eq:intro11} lead to either finite-time blow-up or to small-data global existence. There has been much research investigating such conditions in the uncoupled case: for example, the null condition for a wave equation. In the coupled case, the problem becomes more complicated. We recall the pioneering works in this direction in $n=3$ of \cite{Bachelot,Georgiev}.  For $n=2$ however, the topic of the present paper, the understanding of wave--Klein-Gordon interactions is less complete than $n=3$ and general classification results, apart from some special cases discussed below, are not yet within reach. 

In the present paper we initiate the study of 
quadratic nonlinearities $F_w, F_v$ (coupling the wave and Klein-Gordon fields together), one of which is critical and the other is below-critical. 
Throughout this paper we use ``critical'', short for ``critical-in-time-decay'', to mean  that, assuming $w$ and $v$ obey linear\footnote{Recall that for the uncoupled and linear equations, we have the decay estimates
\begin{equation}\label{eq:RoughDecay}
- \Box w = 0 \Rightarrow |w| \lesssim t^{-\frac{n-1}{2}}, \quad 
- \Box v + m^2 v = 0 \Rightarrow |v| \lesssim m^{-1} t^{-\frac{n}{2}}.
\end{equation}} decay estimates, we have  
\begin{align}\label{intro:crit1}
\lim_{t\to\infty}\int_{t_0}^t \|F\|_{L^2(\RR^n)}\di t = +\infty, \quad \text{and} \quad
\int_{t_0}^t \|(t+r)^{-\delta}F \|_{L^2(\RR^n)}\di t &\leq C, \quad \forall\, t\geq t_0\,, \delta>0.
\end{align}
Analogously, by ``below-critical'' (short for below-critical-in-time-decay) we mean that \eqref{intro:crit1} holds and, for some $\delta>0$,
\be 
\lim_{t\to\infty} \int_{t_0}^t \|(t+r)^{-\delta}F \|_{L^2(\RR^n)}\di t \to + \infty,
\ee
and by ``above-critical'' we mean that
\be 
\lim_{t\to\infty} \int_{t_0}^t \big\|F \big\|_{L^2(\RR^n)}\di t \leq C.
\ee
In this paper we prove the global existence and asymptotic decay of solutions to two model PDEs satisfying the condition that $F_w$ is critical and $F_v$ is below-critical.  Our Model I reads:
\hspace{-1cm}
\begin{align}\notag
-\Box w = v^2, \qquad 
-\Box v + v = P^{\alpha\beta}\del_\alpha w \del_\beta w,
\end{align}
where $P^{\alpha\beta}$ are constants that do \emph{not} satisfy the null condition. Our Model II reads:
\begin{align}\notag
-\Box \tw = P^\alpha \del_\alpha (\tv^2),\qquad 
-\Box \tv + \tv = \tw^2,
\end{align}
where again $P^\alpha$ are some arbitrary constants. We assume our initial data are compactly supported and sufficiently small in high-regularity Sobolev norms. 

Our study of Models I and II is motivated by the broader goal of better understanding theories in mathematical physics that involve the interaction of massive and massless fields. This includes, but is not limited to, the Dirac-Klein-Gordon equations (e.g.  \cite{DW-2021}) as well as equations like  Klein-Gordon-Zakharov and the Einstein-Klein-Gordon equations. We also believe our results can be used to understand certain effective field theories as recently discussed in \cite{Physics}. 
However, we emphasise that the results of this paper should not be viewed as the complete story, but rather as a stepping stone towards completing our understanding of wave--Klein-Gordon phenomena. We also hope that new methods in our proof will aid in future work studying and classifying wave--Klein-Gordon behaviour in spatial dimension $n=2$.

\subsection{Previous Work}
\paragraph*{Uncoupled Wave and Klein-Gordon Equations.}
We first briefly recall some previous results on wave equations in $n=2$ with quadratic nonlinearities. Consider the two PDEs
\be\label{eq:introblow-up}
-\Box w_1 =(w_1)^2, \qquad - \Box w_2 = (\del_t w_2)^2.
\ee
Both nonlinearities here are below-critical, and small-data finite time blow-up is known to hold by \cite{G81} for the first PDE, for the second one by \cite{Agemi, Zhou}. 
As is well known, additional structure within a nonlinearity, other than just the `power' of the decay, can ensure global-existence. The prime example of such structure is Klainerman's null condition \cite{Klainerman80}. In the simplest case this leads to the PDE
\be\label{eq:NullPDE}
-\Box w = (\del_t w)^2 - (\del_x w)^2 - (\del_y w)^2.
\ee
The $n=3$ case was treated much earlier in \cite{Christodoulou, Klainerman86}.

Finally we turn to Klein-Gordon equations $- \Box v + v= F_v$ for $n=2$. Due to the faster linear decay for Klein-Gordon fields, there are several classical works proving global existence for many $F_v$  (e.g. \cite{ST93, OTT2, Hormander, DFX04}). Blow-up is more difficult to arrange for however, and we note the first work to do is \cite{KT} for a below-critical nonlinearity.

\paragraph{Coupled Wave--Klein-Gordon Equations.}
In terms of the vector-field method, coupled wave--Klein-Gordon equations present new difficulties due to the fact that the scaling vector field $t\del_t + x^a\del_a$ does not commute with the Klein-Gordon operator. In two spatial dimensions, numerous difficulties also arise due to the fact that linear wave and Klein-Gordon fields decay very slowly in such low dimensions, and that Klein-Gordon fields do not enjoy the improved null-direction decay estimates that hold for the wave equation.

In the present paper we study our coupled wave--Klein-Gordon systems using the vector-field method adapted to a hyperboloidal foliation of a forward light cone in Minkowski spacetime. 
This method originates in \cite{Klainerman85, Klainerman93}, see also \cite{Hormander},  in the context of  Klein-Gordon equations, and was later revisited for coupled wave--Klein-Gordon equations in \cite{PLF-YM-book}  under the name of the ``hyperboloidal foliation method'' (see also \cite{Tataru} for for the pioneering work on Strichartz estimates for wave equations in the hyperbolic space).
In this method, the scaling vector field is avoided, which allows for an equal treatment of both wave and Klein-Gordon fields.

The study of coupled wave--Klein-Gordon systems has been the subject of much research, and given this we primarily only refer to works on two spatial dimensions. 
Recently the study of coupled wave and Klein-Gordon systems in $\RR^{1+2}$ using the hyperboloidal foliation method was initiated in \cite{Ma2017b}. This has led to  global existence results  for a large class of PDE systems   including null forms and other interesting nonlinearities, under the assumption of compact initial data \cite{ Ma2019, Ma2020,  Dong2006, DM20} and also with no compactness assumption \cite{Dong2005, DongMa, DongMaYuan}. 
There have also been other works \cite{Stingo, IStingo} that have investigated quasilinear wave and Klein-Gordon systems under the null condition. By avoiding the hyperboloidal foliation, these results do not need to restrict to compact data. 
A more general result on cubic nonlinearities has been given in  \cite{C21}, and is along the lines of a `classification' result as alluded to in the beginning of this paper. 

\subsection{Main Statements and Difficulties}\label{sec:introDifficulties}
We can now state our first result for the Model I PDE system. 

\begin{theorem}\label{thm1}
Consider in $\RR^{1+2}$ the coupled wave--Klein-Gordon system
\begin{subequations}\label{eq:model-1}
\begin{align}
-\Box w &= v^2, \label{eq:model-1a} \\
-\Box v + v &= P^{\alpha\beta}\del_\alpha w \del_\beta w,\label{eq:model-1b}
\\
\big(w, \del_t w, v, \del_t v \big) (t_0) &= (w_0, w_1, v_0, v_1). \notag
\end{align}
\end{subequations}
 Let $N\geq 10$ be an integer. There exists $\eps_0>0$ such that for all $0<\eps<\eps_0$ and all compactly supported initial data satisfying the smallness condition
\bea \label{eq:thm1data}
\|w_0\|_{H^{N+1}(\mathbb{R}^2)} + \|w_1\|_{H^{N}(\mathbb{R}^2)}
+ \|v_0\|_{H^{N+1}(\mathbb{R}^2)} + \|v_1\|_{H^{N}(\mathbb{R}^2)}
\leq \eps,
\eea
the Cauchy problem \eqref{eq:model-1} admits a global-in-time solution $(w, v)$. Furthermore there exists a large constant $C>0$ such that the solution satisfies the following pointwise decay estimates
\bea \label{thm1:est}
|w(t,x)| &\leq C\eps t^{-1/2}(t-r)^{1/2}, \quad &|\del w(t, x)| &\leq C \eps t^{-1/2}(t-r)^{-1/2}, \\
|v(t,x)| &\leq C\eps t^{-1}, \quad &|\del v(t,x)| &\leq C\eps t^{-1}.
\eea
\end{theorem}

The major issue with equations \eqref{eq:model-1} is that the best we can expect for the nonlinearities (in the flat constant $t$ slices) is
\bel{eq:666}
\| P^{\alpha\beta}\del_\alpha w \del_\beta w\|_{L^2(\RR^2)} \lesssim t^{-1/2}, \quad \| v^2\|_{L^2(\RR^2)}\lesssim t^{-1}.
\ee
Thus, one quantity at the \emph{borderline} of integrability and the other \emph{strictly below} the borderline of integrability. To the authors' knowledge, such a situation for two dimensional coupled wave--Klein-Gordon equations has \emph{not} been studied before, and a direct application of  vector-field method ideas from classical works would not suffice. Moreover, for all coupled wave--Klein-Gordon equations which are known to admit small smooth global solutions, their nonlinearities behave no worse than critically. For instance, if the nonlinear term $P^{\alpha\beta}\del_\alpha w \del_\beta w$ satisfies the null condition, which is already a hard problem to study in two spatial dimensions, then the Model I problem admits a small global solution as shown in \cite{Ma2020}. 

\paragraph{Novelties in the Proof.} Similar to \eqref{eq:666}, in the hyperboloidal foliation with hyperboloidal time $s = \sqrt{t^2 - r^2}$, the best behaviour we can expect for the nonlinearities is 
\bea\notag
\| (s/t) \del w \|_{L^2_f(\Hcal_s)} &\lesssim 1,
\qquad
&|\del w| &\lesssim t^{-1/2} (t-r)^{-1/2} \lesssim s^{-1},
\\
\| v \|_{L^2_f(\Hcal_s)} &\lesssim 1,
\qquad
&|v| &\lesssim t^{-1}.
\eea
Here $\Hcal_s$ are constant $s$-surfaces defined in \eqref{def:prelioms} and $L^2_f(\Hcal_s)$ is defined in \eqref{flat-int}.
Naive calculations lead us to the estimates
\bea\notag
\| P^{\alpha\beta}\del_\alpha w \del_\beta w\|_{L^2_f(\Hcal_s)} 
&\lesssim \| (s/t) \del w \|_{L^2_f(\Hcal_s)} \| (t/s) \del w \|_{L^\infty(\Hcal_s)}
\lesssim 1,
\\
\| v^2\|_{L^2_f(\Hcal_s)} 
&\lesssim \| v\|_{L^2_f(\Hcal_s)} \| v\|_{L^\infty(\Hcal_s)} 
\lesssim s^{-1}.
\eea
We see one quantity is at the borderline of integrability and the other is below the borderline of integrability.
Furthermore the transformation $
V = v-P^{\alpha\beta}\del_\alpha w \del_\beta w$
which we would like to use to improve our control on $v$, cannot be performed at the highest order of regularity due to top-order regularity issues with the $(\del w)^2$ term. 
Thus for the highest order of regularity in our bootstrap argument new ideas are required.

One of our key ideas is to use a $(t-r)$-weighted energy estimate for the wave equation in \eqref{eq:model-1} (see Proposition \ref{prop:GhostWeight}). Such a weighted energy estimate derives from work of Alinhac  \cite{ Alinhac01b} and has recently been applied to coupled wave--Klein-Gordon systems and the hyperboloidal foliation by the first author in \cite{Dong2006}. Using this, we find
\bea\notag
\| P^{\alpha\beta}\del_\alpha w \del_\beta w\|_{L^2_f(\Hcal_s)} 
&\lesssim \| (s/t) (t-r)^{-1} \del w \|_{L^2_f(\Hcal_s)} \| (t/s) (t-r) \del w \|_{L^\infty(\Hcal_s)}
\lesssim 1,
\\
\| (t-r)^{-1} v^2\|_{L^2_f(\Hcal_s)} 
&\lesssim \| v\|_{L^2_f(\Hcal_s)} \| (t-r)^{-1} v\|_{L^\infty(\Hcal_s)} 
\lesssim s^{-2}.
\eea
Now we have one quantity which is $s^{-1}$ above the borderline of integrability, and the other which is $s^{-1}$ below the borderline of integrability. The situation for $(t-r)$-weighted $L^2$ terms is much better than the unweighted case, and so this gives us hope to close the bootstrap argument for the highest order energy cases provided that we can first show sharp $L^\infty$ estimates for derivatives of the wave component. 

Indeed obtaining sharp pointwise decay estimates on $|\del w|$ is another major issue with the equations \eqref{eq:model-1}, which is resolved by Proposition \ref{prop:RefinedWeightedLinfty}. Indeed Proposition \ref{prop:RefinedWeightedLinfty} is stated in a more general form as  its proof is interesting and we hope it can be applied to other equations. The proof is based on an $L^\infty-L^\infty$ pointwise estimate for undifferentiated waves by Ma \cite{Ma2020}. We then combine both a conformal energy estimate and good commutation properties between the conformal Killing vector field and vector fields tangential to the hyperboloids.
Finally, for the lower order energy cases the aforementioned transformation from $v$ to $V$ is sufficient.

We now turn to our second result for the Model II PDE system. 

\begin{theorem}\label{thm2}
Consider in $\mathbb{R}^{1+2}$  the coupled wave--Klein-Gordon system 
\begin{subequations}\label{eq:model-2}
\begin{align}
-\Box \tw &= P^\alpha \del_\alpha (\tv^2), \label{eq:model-2a}
\\
-\Box \tv + \tv &= \tw^2, \label{eq:model-2b}
\\
\big(\tw, \del_t \tw, \tv, \del_t \tv \big) (t_0) &= (\tw_0, \tw_1, \tv_0, \tv_1),\notag
\end{align}
\end{subequations}
where $P^\alpha$ are arbitrary constants. Let $N\geq 8$ be an integer. There exists $\eps_0>0$ such that for all $0<\eps<\eps_0$ and all compactly supported initial data satisfying the smallness condition
\bea \label{eq:thm2data}
\|\tw_0\|_{H^{N+1}(\mathbb{R}^2)} + \|\tw_1\|_{H^{N}(\mathbb{R}^2)}
+ \|\tv_0\|_{H^{N+1}(\mathbb{R}^2)} + \|\tv_1\|_{H^{N}(\mathbb{R}^2)}
\leq \eps,
\eea
the Cauchy problem \eqref{eq:model-2} admits a global-in-time solution $(\tw, \tv)$. Furthermore there exists a large constant $C>0$ and $\delta'\ll1$ such that the solution satisfies the pointwise decay estimates 
\bea \label{thm2:est}
|\tw(t,x)| &\leq C\eps t^{-1/2}(t-r)^{-1/2+\delta'}, \quad &|\del \tw(t, x)| &\leq C \eps t^{-1/2}(t-r)^{-1/2}, \\
|\tv(t,x)| &\leq C\eps t^{-1}, \quad &|\del \tv(t,x)| &\leq C\eps t^{-1}.
\eea
\end{theorem}
Similar to Model I, the best we can expect for the nonlinearities is
\be\notag
\| \tw^2 \|_{L^2(\RR^2)} \lesssim t^{-1/2}, \quad \|\del(\tv^2)\|_{L^2(\RR^2)}\lesssim t^{-1}.
\ee
The divergence structure appearing in \eqref{eq:model-2a} makes this problem a lot more tractable (indeed, we do not require Proposition \ref{prop:RefinedWeightedLinfty}) and so we defer further details to Section \ref{sec:Model2}. Nevertheless we emphasise that such divergence structure appears naturally in physics, such as for  massless Dirac fields and the Klein-Gordon--Zakharov equations (e.g. \cite{OTT, Dong2006, DW-2021}).

\paragraph{Loss of Decay.}
Following \cite{L-R-annals}, consider momentarily in three spatial dimensions the system
\bea\label{eq:WeakNullPDE}
-\Box w = (\del_t u)^2 ,
\quad 
-\Box u = (\del_t w)^2 - (\del_x w)^2- (\del_y w)^2- (\del_z w)^2.
\eea
If one swaps the nonlinearities in \eqref{eq:WeakNullPDE}, then the equation for $u$ leads to finite-time blow-up \cite{John2}. The ideas in \cite{L-R-annals}  can be used to show that all sufficiently small, smooth and compactly supported initial data lead to global solutions under \eqref{eq:WeakNullPDE}. Moreover the $u$ component will obey the linear decay estimate of $|\del u| \lesssim t^{-1}$ while the $w$ component will suffer from a loss of decay $|\del w|\lesssim t^{-1} \log (1+t)$ due to the absence of the null condition in its nonlinearity $(\del_t u)^2$.

In a similar way, if the nonlinearities $F_w, F_v$ in our Model I were to be swapped (so \eqref{eq:model-1a} $\leftrightarrow$ \eqref{eq:model-1b}), then we would in fact get finite-time blow-up for the resulting wave equation. Moreover, analogously to $w$ in the previous paragraph (i.e., in \eqref{eq:WeakNullPDE}), we get a loss of decay in the region $|t-r|\leq C$ ($C>0$ is a constant) close to the light cone for the Klein-Gordon component $v$.

To understand what we mean here, recall from \cite{Klainerman93, Hormander} that a linear Klein-Gordon field can decay as $
|v| \lesssim \big(s/t \big)^\ell t^{-1}$
for all $\ell\in \mathbb{N}_{\geq 0}$, as long as the initial data are sufficiently smooth. Since $(s/t) \sim t^{-1/2}$ near the light cone $t=r$, this implies a linear decay rate 
\be \label{eq:introkgd}
|v| \lesssim t^{-1-\ell/2}.
\ee

Next, we see from Lemma \ref{lem:null} that  null nonlinear terms $\eta^{\alpha\beta} \del_\alpha w \del_\beta w$ roughly speaking enjoy an extra $(s/t)$ decay compared with general $P^{\alpha\beta} \del_\alpha w \del_\beta w$. Thus, due to the absence of the null condition in the nonlinearity \eqref{eq:model-1b} and the estimates \eqref{thm1:est}, the Klein-Gordon field $v$ cannot decay as fast as \eqref{eq:introkgd}.
A similar loss of decay result also holds for Model II. 

Furthermore the above arguments explain  why the wave nonlinearity \eqref{eq:model-1a} can be generalised to  $P_1 v^2 + P_2 \del_\alpha v \del^\alpha v$, with $P_1, P_2$ constants, but not to $P^{\alpha\beta}\del_\alpha v \del_\beta v$  (indeed using our current strategy the only issues arise when closing the highest order energy).

\paragraph{Outline.} The rest of the paper is organised as follows. In Section \ref{sec:prelim}, we introduce some notation regarding the hyperboloidal foliation and give a brief outline of quadratic semilinear null forms. In Section \ref{sec:Tools}, we present various technical tools required for our later bootstrap arguments (energy estimates, pointwise estimates, etc.).  Finally, in Sections \ref{sec:Model1} and \ref{sec:Model2} we present the proofs of Theorems \ref{thm1} and \ref{thm2} respectively. 

\section{Preliminaries}\label{sec:prelim}

We first state some notation concerning the hyperboloidal foliation method, taken from \cite{PLF-YM-book}. We adopt the signature $(-, +, +)$ in the $(1+2)$--dimensional Minkowski spacetime $(\mathbb{R}^{1+2}, \eta=\text{diag}(-1,1,1)$, and for the point $(t, x) = (x^0, x^1, x^2)$ in Cartesion coordinates we denote its spatial radius by $r \define | x | = \sqrt{(x^1)^2 + (x^2)^2}$. We write $\del_\alpha=\del_{x^\alpha}$ (for $\alpha=0, 1, 2$) for partial derivatives and $
L_a \define x_a \del_t + t \del_a$, where $a= 1, 2$, for the Lorentz boosts. With our sign convention, $x_1=x^1, x_2=x^2$. We denote the scaling vector field by $
L_0 = t\del_t + x^1 \del_1 + x^2 \del_2$. 

We define the following subsets of Minkowski spacetime 
\bea \label{def:prelioms}
\Kcal&\define \{(t, x): r< t-1 \},
\quad 
\del\Kcal\define\{(t, x): r= t-1 \},
\\
\Hcal_s&\define \{(t, x): t^2 - r^2 = s^2 \}, \quad s>1,
\\
\Kcal_{[s_0, s_1]} &\define \{(t, x)\in\Kcal: s_0^2 \leq t^2- r^2 \leq s_1^2\},
\\
\del\Kcal_{[s_0, s_1]}&\define \{(t, x): s_0^2 \leq t^2- r^2 \leq s_1^2\,, \, r=t-1 \}.
\eea
Throughout the paper, we consider functions defined in the interior of the future light cone $\Kcal$, with vertex $(1, 0, 0)$ and boundary $\del\Kcal$. We will consider hyperboloidal hypersurfaces $\Hcal_s$ with $s>1$ foliating the interior of $\Kcal$. We define $\Kcal_{[s_0, s_1]}$ to denote subsets of $\Kcal$ limited by two hyperboloids $\Hcal_{s_0}$ and $\Hcal_{s_1}$ with $s_0 \leq s_1$, and let $\del\Kcal_{[s_0, s_1]}$ denote the conical boundary. 

The semi-hyperboloidal frame is defined by
\bel{eq:semi-hyper}
\underdel_0\define \del_t, \qquad \underdel_a\define \tfrac{L_a}{t} = \tfrac{x_a}{t}\del_t+ \del_a.
\ee
Note that the vectors $\underdel_a$ generate the tangent space to the hyperboloids. Likewise we have the following relation
$$
\del_a = \underline{\del}_a - {x_a \over t} \del_t.
$$
We also introduce 
$$\underdel_\perp\define \del_t+\tfrac{x^a}{t}\del_a = \tfrac{L_0}{t}$$
which is orthogonal to the hyperboloids.


We note the identities
\bel{deriv-identities}
\del_t
=
\tfrac{t^2}{s^2} \big( \underdel_\perp -\tfrac{x^a}{t} \underdel_a \big),
\qquad
\del_a
=
-\tfrac{t x^a}{s^2} \underdel_\perp + \tfrac{x^a x^b}{t^2} \underdel_b + \underdel_a,
\ee
which can be used to deduce pointwise decay of $\del_\alpha w$ for a wave component $w$. We also recall the following decomposition of the flat wave operator in the semi-hyperboloidal frame:  
\bea \label{eq:wavedecomp}
-\Box = \del_t^2 - \sum_{a=1}^2 \del_a^2 
&= \left(\tfrac{s}{t}\right)^2 \del_t \del_t +2\tfrac{x^a}{t} \underline{\del}_a\del_t
-\sum_{a=1}^2 \underline{\del}_a\underline{\del}_a 
- \tfrac{r^2}{t^3}\del_t+\tfrac{2}{t}\del_t
\\
&= \left(\tfrac{s}{t}\right)^2 \del_t \del_t 
+t^{-1} \big( 2\tfrac{x^a}{t} L_a \del_t
-\sum_{a=1}^2 L_a \underline{\del}_a 
- \tfrac{r^2}{t^2}\del_t+2\del_t \big).
\eea
Finally we recall the rotation vector field $$\Omega_{12}\define x^1 \del_2 - x^2 \del_1 = \tfrac{x^1}{t} L_2 - \tfrac{x^2}{t}L_1.$$
Thanks to the second equality above, and our assumption of compact support, we do not need to use $\Omega_{12}$ in this paper. 

\paragraph*{Standard notation.}
Throughout the paper, we use $A\lesssim B$ to denote that there exists a generic constant $C>0$ such that $A\leq BC$.  Spacetime indices are represented by Greek letters $\alpha, \beta, \gamma\in\{0,1,2\}$ while spatial indices are denoted by Roman letters $a, b, c\in\{1,2\}$. We adopt the Einstein summation convention. For the ordered set $\{ \Gamma_i\}_{i=1}^5\define\{ \del_{0},\del_1, \del_2, L_1,L_2\}$, and for a multi-index $I=(\alpha_1, \ldots, \alpha_5)$ of length $|I|\define \alpha_1 + \ldots \alpha_5 =  m$ we write the product of vector-fields by $Z^I\define\Gamma_1^{\alpha_1}\Gamma_2^{\alpha_2}\Gamma_3^{\alpha_3}\Gamma_4^{\alpha_4}\Gamma_5^{\alpha_5}$.

\subsection{Semilinear null structure}\label{sec:Transf}
In this subsection we briefly remind the reader about properties of semilinear null forms. For $\phi, \psi$ some smooth functions, the null forms are:
\bea\label{eq:null00}
Q_0(\phi,  \psi) &\define \eta^{\alpha\beta} \del_\alpha \phi \del_\beta \psi, \quad
Q_{\mu\nu}( \phi,  \psi) \define \del_\mu \phi \del_\nu \psi - \del_\nu \phi \del_\mu \psi.
\eea
 These were first identified by Klainerman in \cite{Klainerman80}. 
We have the identities
\bea\label{eq:NullFormDecomps}
Q_0( \phi, \psi) &= t^{-1} \big( L_0 \phi\cdot \del_t \psi-\sum_{a=1}^2\del_a \phi \cdot L_a \psi\big)
\\
Q_{ab}(\phi,  \psi) &= t^{-1} \left( \del_a \phi\cdot L_b \psi- \del_b \phi \cdot L_a \psi + \Omega_{ab}\phi\cdot\del_t \psi\right),
\\
Q_{0a}(\phi, \psi) &= t^{-1} \left( \del_t \phi\cdot L_a \psi-L_a \phi \cdot \del_t \psi\right).
\eea

The null forms $Q_{ab}$ and $Q_{0a}$ are sometimes called strong null forms since the scaling vector field $L_0$ does not appear in the decomposition \eqref{eq:NullFormDecomps}. As identified by Georgiev \cite{Georgiev}, strong null forms obey good estimates. In particular:
\be\notag
\aligned
\big| Q_{\mu\nu}(\phi,\psi)\big| 
\lesssim t^{-1} \big(\sum_\alpha |\del_\alpha  \phi||Z \psi|  + \sum_\alpha |Z\phi||\del_\alpha \psi|\big).
\endaligned
\ee

Estimates on the $Q_0(f,g)$ null form depend on whether $f$ and $g$ obey wave or Klein-Gordon equations. If both obey wave equations, then we obtain good estimates from the identity
\be\notag
|Q_0(\widetilde{w}, w)| = |\del_\alpha \widetilde{w}\del^\alpha w |
\lesssim  t^{-1} \big( \sum_\alpha |\del_\alpha \widetilde{w}||L_0 w| + \sum_{a, \alpha} |L_a \widetilde{w} ||\del_\alpha w |\big),
\ee
and the fact that  $|L_0 w|$ can be controlled by the conformal energy. 
If instead we have a wave--Klein-Gordon interaction, then we can arrange for the $L_0$ derivative to \emph{only} act on the wave component, leading to the identity:
\be\notag
\big|Q_0( v, w)\big|
= \big| \del_\alpha v \del^\alpha w \big|
\lesssim  t^{-1}  \big( \sum_\alpha \big| \del_\alpha v|| L_0 w\big| + \sum_{a, \alpha} \big| L_a v| |\del_\alpha w \big| \big).
\ee
Such structure has been exploited by the first author in \cite{Dong2005}, and good estimates can be concluded by controlling $|L_0 w|$, for example by the conformal energy.
We cannot, however, at present gain an extra $t^{-1}$ decay on $Q_0( v,  \tilde{v})$ (product of two Klein-Gordon components) since we do not have good bounds on $L_0 v$.

\section{Auxiliary tools}\label{sec:Tools}
\subsection{Standard energy estimates}

Following \cite{PLF-YM-book}, we first introduce the energy functional $E_m$, in the Minkowski background, for a function $\phi$ defined on a hyperboloid $\Hcal_s$: 
\bel{eq:2energy} 
\aligned
E_m(s, \phi)
&\define\int_{\Hcal_s} \Big( \big( (s/t)\del_t \phi \big)^2+ \sum_a \big(\underdel_a \phi \big)^2+ m^2 \phi^2 \Big) \, \di x
                \\
               &= \int_{\Hcal_s} \Big( \big( \underdel_\perp \phi \big)^2+ \sum_a \big( (s/t)\del_a \phi \big)^2+ \big( t^{-1}\Omega_{12} \phi \big)^2+ m^2 \phi^2 \Big) \, \di x.
 \endaligned
 \ee
In the massless case we denote $E(s, \phi)\define E_0(s, \phi)$.
In the above, the integral in $L^1(\Hcal_s)$ is defined from the standard (flat) metric in $\RR^2$, i.e.,
\bel{flat-int}
\|\phi \|_{L^p_f(\Hcal_s)}^p
\define\int_{\Hcal_s}|\phi |^p \, \di x 
=\int_{\RR^2} \big|\phi(\sqrt{s^2+r^2}, x) \big|^p \, \di x,
\qquad
p \in [1, +\infty).
\ee

\begin{proposition}[Energy estimate]\label{prop:BasicEnergyEstimate}
Let $m\geq 0$ and $\phi$ be a sufficiently regular function defined in the region $\Kcal_{[s_0, s]}$, vanishing near $\del\Kcal_{[s_0, s]}$ and satisfying
\be\notag
-\Box \phi + m^2 \phi = f.
\ee
For all $s \geq s_0$, it holds that
\be\notag
E_m(s, \phi)^{1/2}
\leq 
E_m(s_0, \phi)^{1/2}
+ \int_{s_0}^s \| f\|_{L^2_f(\Hcal_{\tau})} \, \di \tau.
\ee

\end{proposition}

\subsection{Weighted energy estimates}
Following ideas of Alinhac \cite{ Alinhac01b}, we have the following weighted energy estimate which for instance have been applied to coupled wave--Klein-Gordon systems in \cite[Prop. 3.2]{Dong2006}.

\begin{proposition}\label{prop:GhostWeight}
For $m\geq 0$ consider a sufficiently regular function $\phi$ defined in the region $\Kcal_{[s_0, s]}$, vanishing near $\del\Kcal_{[s_0, s]}$ and satisfying
\be\notag
-\Box \phi + m^2 \phi = f.
\ee
Then for $\gamma>0$ we have
\bea\notag
&\int_{\Hcal_s} \big(t-r\big)^{-\gamma}\Big[ \big((s/t)\del_t \phi\big)^2 + \sum_a (\underdel_a \phi)^2 + m^2 \phi^2 \Big] \di x 
\\
&
\leq C \int_{\Hcal_{s_0}}\big(t-r\big)^{-\gamma} \Big[ ((s_0/t)\del_t \phi)^2+ \sum_a (\underdel_a \phi)^2 + m^2 \phi^2 \Big] \di x 
+ C \int_{s_0}^s \big\|(\tau/t) (t-r)^{-\gamma} f \del_t \phi \big\|_{L^1_f(\Hcal_\tau)} \di\tau .
\eea
\end{proposition}
\begin{proof}
As shown in \cite{Dong2006}, the proof works by multiplying the PDE by $(t-r)^{-\gamma} \del_t \phi$ and deriving the identity
\bea\notag
&\frac12 \del_t \Big[ (t-r)^{-\gamma} \big( (\del_t \phi)^2 + (\del_a \phi)^2 + m^2 \phi^2 \big) \Big] 
- \sum_a \del_a \Big( (t-r)^{-\gamma} \del_t \phi \del_a \phi \Big)
\\
& + \frac{\gamma}{2}(t-r)^{-\gamma-1} \Big[\sum_a  \Big( \frac{x_a}{r} \del_t \phi + \del_a \phi\Big)^2 + m^2\phi^2 \Big] 
= (t-r)^{-\gamma} f \del_t \phi.
\eea
The first term of the second line is non-negative. 
\end{proof}

\subsection{Conformal energy estimates}
We now introduce a conformal-type energy which was adapted to the hyperboloidal foliation setting by Huang and Ma in three spatial dimensions in \cite{YM-HH} and in two spatial dimensions by Wong \cite{Wong}. A key part of this lemma, due to Ma \cite{Ma2019}, is in giving an estimate for the weighted $L^2$ norm $\| (s/t) \phi \|_{L^2_f(\Hcal_s)}$ for a wave component $\phi$.

\begin{proposition}\label{lem:ConformalEnergy}
Let $\phi$ be a sufficiently regular function defined in the region $\Kcal_{[s_0, s]}$ and vanishing near $\del\Kcal_{[s_0, s]}$. Define the conformal energy
\be\notag
E_{con} (s,\phi)
\define
\int_{\Hcal_s} \Big( \sum_a \big( s \underdel_a \phi \big)^2 + \big( K \phi + \phi \big)^2 \Big) \, \di x,
\ee
where we used the vector field $K \phi 
\define \big( s \del_s + 2 x^a \underdel_a \big) \phi$. 
Then for all $s \geq s_0$ we have the energy estimate
\bel{eq:con-estimate} 
E_{con} (s,\phi)^{1/2}
\leq 
E_{con} (s_0, \phi)^{1/2}
+
2 \int_{s_0}^s \tau \| \Box \phi \|_{L^2_f(\Hcal_{\tau})} \, \di\tau.
\ee
Furthermore we have 
\bel{eq:l2type-wave} 
\|(s/t) \phi \|_{L^2_f(\Hcal_s)} \lesssim \|(s_0/t) \phi \|_{L^2_f(\Hcal_{s_0})} + \int_{s_0}^s \tau^{-1} E_{con}(\tau, \phi)^{1/2} \di \tau.
\ee
\end{proposition}


\subsection{Commutator and null-form estimates}
We first have the following identities
\bea\label{eq:CommutatorLdel}
\left[ L_a, L_b \right] &= \tfrac{x^a}{t}L_b - \tfrac{x^b}{t}L_a , \quad 
[ \del_t, L_a ]=\del_a , \quad 
[\del_b, L_a] = \delta_{ab} \del_t, \\
[t, L_a] &= -x_a , \quad
[x^b, L_a] = -t \delta^b_a.
\eea
By using these identities and writing $L_0 = t\del_t+x^b\del_b$ and $K = t\del_t + x^b \del_b + (x^b/t)L_b$ we find
\bea\label{eq:CommutatorsK}
[\del_\alpha, L_0]&=\del_\alpha, \quad &[L_a, L_0]&=0,
\quad
[L_a, K] &= (s/t)^2 L_a, \quad &[\del_a, K]&=(2/t)L_a.
\eea
Finally we have the useful property that for the $Q_0$ null form 
\bea \label{eq:CommutatorsQzero}
L_a Q_0(f,g) &= Q_0(L_a f,  g) + Q_0( f, L_a g), \quad
\del_\alpha Q_0( f,  g) = Q_0(\del_\alpha f,  g) + Q_0( f, \del_\alpha g) .
\eea

The following lemma allows us to control the commutators. Its proof can be found for instance in \cite[\textsection3]{PLF-YM-book}.
\begin{lemma} \label{lem:est-comm}
Let $\phi$ be a sufficiently regular function supported in the region $\mathcal{K}$. Then, for any multi-index $I$, there exist generic constants $C=C(|I|)$ such that
\begin{subequations}\label{eq:est-cmt}
\begin{align}
\label{eq:est-cmt1}
 \big| [Z^I, \del_\alpha] \phi \big| 
&\leq 
C \sum_{|J|<|I|} \sum_\beta \big|\del_\beta Z^J \phi \big|,
\\
\label{eq:est-cmt2}
 \big| [Z^I, \underdel_a] \phi \big| 
&\leq 
C t^{-1} \sum_{| I' |\leq | I |} \big| Z^{I'}\phi \big| ,
\\
\label{eq:est-cmt4}
 \big| [Z^I, \del_\alpha \del_\beta] \phi \big| 
&\leq 
C \sum_{|J|<|I|} \sum_{\gamma, \gamma'} \big| \del_\gamma \del_{\gamma '} Z^J \phi \big|,
\\
\label{eq:est-cmt6}
 \big|Z^I ((s/t) \del_\alpha \phi) \big| 
&\leq 
|(s/t) \del_\alpha Z^I\phi| + C \sum_{|I'|\leq  | I |,}\sum_\beta \big|(s/t) \del_\beta Z^{I'} \phi \big|.
\end{align}
\end{subequations}
Recall here that Greek indices $\alpha, \beta \in \{0,1,2\}$ and Roman indices $a,b \in \{1,2\}$. 
\end{lemma}

We next state an important estimate for null forms in terms of the hyperboloidal coordinates. The proof is standard and can be found in \cite[\textsection 4]{PLF-YM-book}.
\begin{lemma}
\label{lem:null}
Let $\phi, \psi$ be sufficiently regular spacetime functions supported in the region $\mathcal{K}$, and let $Q(\phi, \psi)$ denote any one of the null forms given in \eqref{eq:null00}. Then there exists a constant $C=C(|I|)>0$ such that 
\bel{eq:est-null1}
\aligned
\big| Z^I Q(f, \psi) \big| &
\leq  
C (s/t)^2 \sum_{\substack{| I_1 | + | I_2 | \leq | I |}} \big|Z^{I_2}\del_t f \cdot Z^{I_2}\del_t \psi \big|
\\
&\quad + C \sum_{\substack{| I_1 | + | I_2 | \leq | I |}}\sum_{a, \beta} \Big( \big| Z^{I_1} \underdel_a \phi \cdot Z^{I_2}  \underdel_\beta \psi \big| 
+ \big| Z^{I_1} \underdel_\beta \phi \cdot Z^{I_2} \underdel_a \psi \big| \Big).
\endaligned
\ee
\end{lemma}

Finally we end with the following short lemma, whose proof can be found in \cite{PLF-YM-book}.
\begin{lemma}\label{Lemma:Est-s/t}
In the cone $\Kcal$, there exists a constant $C>0$, determined by $I$ and $J$, such that
\be\notag
|\del^I L^J(s/t)| \leq 
     \begin{cases}
      C(s/t) &\quad |I|=0,\\
      Cs^{-1} &\quad |I|>0.
     \end{cases}
\ee
\end{lemma}

\subsection{Pointwise Estimates}
We now state a Klainerman-Sobolev estimate in terms of the hyperboloidal coordinates. The proof is standard and can be found in \cite[\textsection 5]{PLF-YM-book}.

\begin{lemma}[Sobolev Estimate] \label{lem:sobolev}
For all sufficiently smooth functions $\phi= \phi(t, x)$ supported in $\Kcal$ and for all  $s \geq 2$, there exists a constant $C>0$ such that
\bel{eq:Sobolev2}
\sup_{\Hcal_s} \big| t \phi(t, x) \big|  \leq C \sum_{| J |\leq 2} \| L^J \phi \|_{L^2_f(\Hcal_s)}.
\ee
Furthermore we have
\be
\sup_{\Hcal_s} \big| s \phi(t, x) \big|  \lesssim \sum_{| J |\leq 2} \| (s/t) L^J \phi \|_{L^2_f(\Hcal_s)},
\ee
\be\label{eq:weightedSobolev}
\sup_{\Hcal_s} \big| s (t-r)^{-\gamma} \phi(t, x) \big|  \lesssim \sum_{| J |\leq 2} \| (s/t) (t-r)^{-\gamma} L^J \phi \|_{L^2_f(\Hcal_s)}.
\ee
\end{lemma}

We next state an $L^\infty-L^\infty$ estimate based on Kirchhoff's formula from \cite[Lemma 3.5]{Ma2020}. 
\begin{lemma}
\label{lem:supwave}
Suppose $\phi$ is a sufficiently smooth function supported in $\Kcal_{[s_0, s]}$,  vanishing near $\del\Kcal$ and satisfying the following Cauchy problem
\be \notag
\aligned
- \Box \phi &= f,
\quad
\phi|_{\Hcal_{s_0}} = \phi_0, \quad
\del_t \phi|_{\Hcal_{s_0}}  =\phi_1,
\endaligned
\ee
with $\phi_0, \phi_1$ being $C_c^\infty$ functions supported in $\Hcal_{s_0} \cap \Kcal$. Suppose that $f$ vanishes near $\del\Kcal$ and within  $\Kcal_{[s_0, s]}$ satisfies the bound
\be\notag
|f(t,x)| \leq \cf t^{-2}.
\ee
Then there exists a $C>0$ such that for all $(t,x)\in \Kcal_{[s_0, s]}$  the following estimate holds
\be\notag
|\phi(t,x)| \leq C \cf (s/t) + \ci s^{-1},
\ee
with $\ci$ a constant determined by $\phi_0$ and $\phi_1$.  
\end{lemma}

We now present a novel method to deduce refined pointwise estimates for a solution to a wave equation under some mild assumptions on the nonlinearity. 

\begin{proposition}\label{prop:RefinedWeightedLinfty}
Let $N_0\in \mathbb{N}_{\geq 2}$.
Suppose $\phi$ is a sufficiently smooth function supported in $\Kcal_{[s_0, s]}$,  vanishing near $\del\Kcal$ and satisfying the following Cauchy problem
\be \label{eq:dec455}
\aligned
- \Box \phi = f,
\quad 
\phi|_{\Hcal_{s_0}} = \phi_0, \quad
\del_t \phi|_{\Hcal_{s_0}}  =\phi_1,
\endaligned
\ee
with $\phi_0, \phi_1$ being $C_c^\infty$ functions supported in $\Hcal_{s_0} \cap \Kcal$.
Suppose that $f$ vanishes near $\del\Kcal$ and for $A, B>0$ some constants and $|I| \leq N_0$ the following bounds hold
\bea\notag
|Z^I f(t,x)| &\leq A t^{-2}, \quad &(t,x) \in \Kcal_{[s_0, s]} \\
 \| Z^I f\|_{L^2_f(\Hcal_{\tau})} &\leq B\tau^{-1}, \quad &\tau \in [s_0, s].
\eea
Then there exists a  constant $C>0$ such that
\bea \notag
\sum_\alpha |s \del_\alpha Z^I w| + \sum_\alpha \big| (t-r) (t/s) \del_\alpha Z^I  w \big| 
& \leq C  D \qquad |I|\leq N_0-2
\eea
where $D$ is a constant depending on $A, B$ and $\phi_0, \phi_1$. 
\end{proposition}
 
\begin{proof}
By Lemma \ref{lem:supwave} and the smallness on the initial data we have, for $|I|\leq N_0$,
\be \notag
|Z^I w| \lesssim A(s/t) + \ci s^{-1}
\lesssim (A+\ci)(s/t).
\ee
In the above $\ci$ is a constant determined by $\phi_0$ and $\phi_1$ however we may increase it below from line to line as needed. 
By commuting $L_a$ through \eqref{eq:dec455} we also have, for $|I|\leq N_0-1$,
\be\notag
|L_a Z^I w|
\lesssim (A+\ci)(s/t).
\ee

We now use these refined pointwise estimates to derive refined pointwise estimates for derivatives of the wave component. 
Using the conformal energy estimate of Lemma \ref{lem:ConformalEnergy}  we have, for $|I|\leq N_0$,  
\bea\notag
E_{con} (s,Z^I w)^{1/2}
&\leq 
E_{con} (s_0,Z^I w)^{1/2}
+
C  \int_{s_0}^s \tau \| Z^I F\|_{L^2_f(\Hcal_{\tau})} \, \di \tau
\lesssim \ci  + B s.
\eea 
By the Klainerman-Sobolev estimate of Lemma \ref{lem:sobolev}, commutator identities in \eqref{eq:CommutatorsK}, and estimates from Lemma \ref{Lemma:Est-s/t}, we have on $\Hcal_s$, for $|I|\leq N_0-2$,  
\bea \notag
|(K+1)Z^I w|
&\lesssim t^{-1} \sum_{|J'|\leq 2} \| L^{J'} ((K+1)Z^I w) \|_{L^2_f(\Hcal_s)}
\lesssim t^{-1} \sum_{|I|\leq N-3}E_{con}(s,Z^I w)^{1/2}
\\
&\lesssim t^{-1}( \ci + B s)
\lesssim ( \ci + B)(s/t).
\eea
Finally by using the identity $Kw=L_0 w + (x^a/t)L_a w$ we have, for $|I|\leq N_0-2$,
\bea \notag
|L_0 Z^I w| & \leq
|(K+1)Z^Iw| + \sum_a |L_aZ^Iw| + |Z^I w|
 \lesssim
(A + B + \ci)(s/t).
\eea
From the identities $\underdel_a= t^{-1} L_a$ and $\underdel_\perp= t^{-1} L_0$, we find, for $|I|\leq N_0-2$,
\be\notag
|\underdel_a Z^I w| + |\underdel_\perp Z^I w|
\lesssim
(A + B + \ci)t^{-1} (s/t) .
\ee
Using the identities \eqref{deriv-identities} we deduce that,  for $|I|\leq N_0-2$,
\be\notag
|\del_t Z^I w| + |\del_a Z^I w|
\lesssim
(A + B + \ci) (t/s)^2 t^{-1} (s/t) 
 \lesssim  (A + B + \ci) s^{-1}.
\ee
Finally, from the identity $s^2 = t^2-r^2$,
\be\notag
\big| (t-r) (t/s) \del_\alpha Z^I  w \big| 
\lesssim (A + B + \ci)\frac{s^2 t}{(t+r)s^2}
\lesssim (A + B + \ci).
\ee
The proof is finished.
\end{proof}

We conclude this section with an estimate on second order derivatives of wave components. It follows easily from  the decomposition \eqref{eq:wavedecomp}, see for example a proof in \cite[\textsection8.1, \textsection8.2]{PLF-YM-book}).

\begin{lemma}\label{lem:Wavedeldel}
Suppose $\phi$ is a sufficiently smooth function supported in $\Kcal_{[s_0, s]}$  vanishing near $\del\Kcal$ and satisfying
\be \notag
\Box \phi = f.
\ee
Then there exists a constant $C>0$ such that
\be \notag
\sum_{\alpha, \beta} |\del_\alpha \del_\beta \phi| \leq C\sum_\alpha \Big(  \frac{1}{t-r} \left( |\del_\alpha L \phi| + |\del_\alpha \phi| \right) + \frac{t}{t-r} |f| \Big).
\ee
\end{lemma}

\section{Global Existence}\label{sec:Model1}
Our model problem I reads
\begin{subequations}\label{eq:model-1rep}
\begin{align}
-\Box w &= v^2, \label{eq:model-1-wave}
\\
-\Box v + v &= P^{\alpha\beta}\del_\alpha w \del_\beta w, \label{eq:model-1-KG}
\end{align}
\end{subequations}
with initial data $\big(w, \del_t w, v, \del_t v \big) (t_0) = (w_0, w_1, v_0, v_1)$ and where the constant coefficients $P^{\alpha\beta}$ do \emph{not} satisfy the null condition. 

\subsection{The bootstraps and preliminary estimates}
Fix $N\in\mathbb{Z}$ a large integer. As shown in \cite[\textsection11]{PLF-YM-book}, initial data posed on the hypersurface $\{t_0=2\}$ and localised in the unit ball $\{x\in\RR^2:r\leq 1\}$ can be developed as a solution of the PDE to the initial hyperboloid $\Hcal_{s_0}$, $s_0=2$, with the smallness conserved. Thus there exists a constant $C_0>0$ such that on the initial hyperboloid $\Hcal_{s_0}$ the following energy bounds hold for all $|I|\leq N$:
\bea\label{eq:m1BApre}
E(s_0, Z^I w)^{1/2} + \big\|  (s_0/t) (t-r)^{-1}\del Z^I w \big\|_{L^2_f(\Hcal_{s_0})} + E_1 (s_0, Z^I v)^{1/2} \leq C_0 \eps. 
\eea

Next we fix $\delta, \eps, C_1$ positive constants such that $\delta\ll 1$. Consider, for $|I|\leq N-1$
\begin{subequations} \label{eq:BA-Easy}
\be
s^{-\delta} E(s, Z^I w)^{1/2} + \big\| \tfrac{s}{t} (t-r)^{-1}\del Z^I w \big\|_{L^2_f(\Hcal_s)} + E_1 (s, Z^I v)^{1/2}
\leq C_1 \eps 
\ee
and for $|I| = N$
\be 
s^{-\delta} \big\|  (s/t) (t-r)^{-1}\del Z^I w \big\|_{L^2_f(\Hcal_s)} + s^{-(1+\delta)} E_1 (s, Z^I v)^{1/2}
\leq C_1 \eps 
\ee
\end{subequations}
 \underline{For the rest of section \ref{sec:Model1}} we assume, without restating the fact, that \eqref{eq:BA-Easy} hold on a hyperbolic time interval $[s_0, s^*)$  where
 $s^* \define \sup_{s\geq s_0} \{\eqref{eq:BA-Easy} \text{ hold on }  [s_0, s)\}$. 
Taking $C_1 \gg C_0$ we see, by continuity of the above $L^2$ functionals, that $s^*>s_0$. 
We first assume $s^*<+\infty$ and then we will derive a contradiction to assert that $s^*=+\infty$. 

The bootstrap assumptions \eqref{eq:BA-Easy} and  definition \eqref{eq:2energy} imply the following 
\bea
\| Z^I v\|_{L^2_f(\Hcal_s)} + \| (s/t) \del Z^I v\|_{L^2_f(\Hcal_s)} + s^{-\delta} \| (s/t) \del Z^I w\|_{L^2_f(\Hcal_s)}
&\lesssim C_1 \eps ,
\qquad
|I|\leq N-1,
\\
\| Z^I v\|_{L^2_f(\Hcal_s)} + \| (s/t) \del Z^I v\|_{L^2_f(\Hcal_s)} + s \big\|  (s/t) (t-r)^{-1}\del Z^I w \big\|_{L^2_f(\Hcal_s)} 
&\lesssim C_1 \eps s^{1+\delta},
\qquad
|I| = N.
\eea
The bootstrap assumptions \eqref{eq:BA-Easy} and the Sobolev estimates of Lemma \ref{lem:sobolev} imply the following  pointwise Klein-Gordon estimates
\begin{align} 
| Z^I v| 
\lesssim &C_1 \eps t^{-1},
\qquad
|I| \leq N-3, \label{eq:m1pointwise-KG1}
\end{align}
as well as the following pointwise wave estimate
\begin{align} 
| \del Z^I w| 
\lesssim &C_1 \eps s^{-1+\delta},
\qquad
|I| \leq N-3.\label{eq:m1pointwise-wave1}
\end{align}
Using \eqref{eq:m1pointwise-KG1} we have, for $|I|\leq N-3$,
\bea\notag
|Z^I(v^2)|
&\lesssim (C_1 \eps)^2 t^{-2}, \quad 
\\
\| Z^I (v^2)\|_{L^2_f(\Hcal_{s})} &\lesssim (C_1 \eps)^2 s^{-1}.
\eea
Proposition \ref{prop:RefinedWeightedLinfty} then yields, for $|I|\leq N-5$,
\bea \label{eq:RefinedWeightedLinfty}
&|\underdel_a Z^I w| + |\underdel_\perp Z^I w|
 \leq C
 C_1 \eps t^{-1} (s/t), 
\\
&s |\del_t Z^I w| + s|\del_a Z^I w| + \big| (t-r) (t/s) \del_\alpha Z^I  w \big| 
 \leq C  C_1 \eps.
\eea

\subsection{Lower-order bootstraps}
\subsubsection{Klein-Gordon component}
We first introduce the new variable $V \define v-P^{\alpha\beta}\del_\alpha w \del_\beta w$
which obeys the equation
\be \label{e1:m1transfV}
-\Box V + V
= 2 P^{\alpha\beta}Q_0(\del_\alpha w, \del_\beta w) 
-2vP^{\alpha\beta}\del_\alpha v \del_\beta w - 2 P^{\alpha \beta}v \del_\alpha w \del_\beta v.
\ee
We begin with a Lemma estimating the final two nonlinearities appearing in \eqref{e1:m1transfV}. 
\begin{lemma}\label{lem:model1-300}
For all multi-indices of order $|I|\leq N-1$, there exists a $C>0$ such that  
\bea\notag 
\sum_{\alpha, \beta} \| Z^I (v \cdot \del_\alpha v \cdot \del_\beta w)\|_{L^2_f(\Hcal_{s})}
\leq
C (C_1 \eps)^3 s^{-1-\delta}.
\eea
\end{lemma}

\begin{proof}
If $N$ is sufficiently large so that we can apply the estimates \eqref{eq:m1pointwise-KG1} and  \eqref{eq:m1pointwise-wave1} ($\frac{N-1}{2}+1\leq N-3$ suffices), then we find using the standard commutator estimates of Lemma \ref{lem:est-comm}, that 
\begin{align*}
& \| Z^I (v \del_\alpha v \del_\beta w)\|_{L^2_f(\Hcal_{s})}
\\
&\lesssim
\sum_{|I_1|\leq \frac{N-1}{2}, |I_2| \leq \frac{N-2}{2}} \|(t/s)Z^{I_1}v \cdot Z^{I_2}\del_\alpha v \|_{L^\infty(\Hcal_{s})} \sum_{|I_2|\leq N-1} \| (s/t)Z^{I_2}  \del_\beta w\|_{L^2_f(\Hcal_{s})}
\\
& \quad + \sum_{|I_1|\leq \frac{N-1}{2}, |I_2| \leq \frac{N-2}{2}}\|Z^{I_1} \del_\alpha v \cdot Z^{I_2} \del_\beta w \|_{L^\infty(\Hcal_{s})} \sum_{|I_2|\leq N-1} \| Z^{I_2}v\|_{L^2_f(\Hcal_{s})}
\\
& \quad  + \sum_{|I_1|\leq \frac{N-1}{2}, |I_2| \leq \frac{N-2}{2}} \|(t/s)Z^{I_1}v \cdot Z^{I_2} \del_\beta w \|_{L^\infty(\Hcal_{s})}\sum_{|I_2|\leq N-1} \|(s/t) Z^{I_2} \del_\alpha v\|_{L^2_f(\Hcal_{s})}
\\
& \lesssim	
	(C_1 \eps)^3  \Big(\|(t/s)t^{-2} s^\delta \|_{L^\infty(\Hcal_{s})}
+\|t^{-1} s^{-1+\delta}\|_{L^\infty(\Hcal_{s})} 
+\|(t/s)t^{-1}s^{-1+\delta} \|_{L^\infty(\Hcal_{s})} \Big)
\\
&\lesssim 
(C_1\eps)^3 s^{-2+\delta}.
\end{align*}
The required conclusion holds provided $\delta \leq 1/2$. 
\end{proof}

\begin{proposition}\label{prop:model1-343}
For all multi-indices of order $|I|\leq N-1$, there exists a $C>0$ such that  
\bea \notag
E_1 (s, Z^I V)^{1/2}
\leq
C \eps + C (C_1 \eps)^2.
\eea
\end{proposition}

\begin{proof}
We first note that by the smallness assumptions on the data $
E_1(s_0, Z^I V)^{1/2} \leq C \eps$. 
Next, by the energy estimate of Proposition \ref{prop:BasicEnergyEstimate} applied to equation \eqref{e1:m1transfV}, we have
\bea \label{eq:M1energyde1}
E_1(s, Z^I V)^{1/2}
&\lesssim
E_1(s_0, Z^I V)^{1/2}
\\
&
+ \int_{s_0}^s 
\Big(\| Z^I Q_0( P^{\alpha\beta}\del_\alpha w,  \del_\beta w) \|_{L^2_f(\Hcal_{\tau})}
+ \| Z^I(v P^{\alpha\beta}\del_\alpha v \del_\beta w)\|_{L^2_f(\Hcal_{\tau})}\Big)
 \, \di \tau.
\eea
Clearly the second term under the integral  in \eqref{eq:M1energyde1} can be estimated using Lemma \ref{lem:model1-300}.  
For the first term  we use the null-form estimates of Lemma \ref{lem:null} which imply 
\bea\label{eq:M1nullstuff} 
|  Z^I &Q_0( P^{\alpha\beta}\del_\alpha w, \del_\beta w) |
 \lesssim \sum_{\substack{| I_1 | + | I_2 | \leq | I | }}  \Big[
(s/t)^2 \sum_{\alpha,\beta} \big|Z^{I_1} \del_t \del_\alpha w \cdot Z^{I_2}  \del_t \del_\beta w \big|
\\
&
\quad + \sum_{a, \alpha, \beta} \big| Z^{I_1} \underdel_a \del_\alpha w \cdot Z^{I_2} \del_t \del_\beta w \big| 
+ \sum_{a, b, \alpha, \beta} \big| Z^{I_1} \underdel_a \del_\alpha w \cdot Z^{I_2} \underdel_b\del_\beta w \big| \Big].
\eea
We begin by studying the two types of terms appearing on the right hand side of \eqref{eq:M1nullstuff}. Using Lemma \ref{lem:Wavedeldel}, and the commutator estimates of  \eqref{eq:CommutatorLdel} and \eqref{eq:est-cmt1} we have, 
\bea \label{eq:m1-998}
\big|(s/t) Z^I \del_t \del_\alpha w \big| 
&\lesssim \sum_{|J|\leq |I|} \sum_{\beta, \gamma} \big|(s/t)\del_\beta \del_\gamma Z^J w \big| 
\\
&\lesssim 
\frac{1}{(t-r)}
	\sum_{\substack{|J|\leq |I|+1}} \big|(s/t) \del Z^J w \big| 
+ \frac{s}{(t-r)} \sum_{\substack{|J|\leq \frac{|I|}{2}}} \big|Z^Jv \big| 
	\sum_{\substack{|J'|\leq |I|}} \big|Z^{J'}v \big|.
\eea

The second type of term appearing on the right hand side of \eqref{eq:M1nullstuff} is easier to treat. 
In particular, by the commutator estimates \eqref{eq:CommutatorLdel}, \eqref{eq:est-cmt1} and \eqref{eq:est-cmt2} we have,
\bea \label{eq:m1-999}
\big| Z^I \underdel_a \del_\alpha w\big| 
&\lesssim 
	\sum_{\substack{|J|\leq |I|}} 
t^{-1} \big| L_a Z^J \del_\alpha w\big| 
+ t^{-1} \sum_{\substack{|J|\leq |I|}}
\big| Z^J\del_\alpha w\big| 
\lesssim 
t^{-1}\sum_{|J|\leq |I|+1, \beta} 
 \big| \del_\beta Z^J w\big|.
\eea

We now consider the case $|I|\leq N-1$.
If $N$ is sufficiently large to apply the estimates \eqref{eq:m1pointwise-KG1} and  \eqref{eq:m1pointwise-wave1} ($\tfrac{N+1}{2}\leq N-2$ suffices), then using \eqref{eq:m1-998} we find, for $|I|\leq \tfrac{N-1}{2}$,
\bea\notag
\big|(s/t)Z^I \del_t \del_\alpha w \big| 
&\lesssim 
\frac{1}{(t-r)}
	\sum_{|J|\leq \frac{N+1}{2}}  \big|(s/t) \del Z^Jw \big| 
+ \frac{s}{(t-r)} \Big( \sum_{\substack{|J|\leq \frac{N-1}{2}}} \big|Z^Jv \big| \Big)^2
\\
&\lesssim
	 \frac{s/t}{(t-r)} C_1\eps s^{-1+\delta} + \frac{s}{(t-r)} (C_1\eps t^{-1})^2
\lesssim
C_1 \eps s^{-2+\delta},
\eea
while using \eqref{eq:m1-999} we find, for $|I|\leq \tfrac{N-1}{2}$,
\bea\notag
\big|Z^I\underdel_a \del_\alpha w\big| 
&\lesssim 
	\sum_{|J|\leq \frac{N+1}{2}, \beta} 
t^{-1} \big| \del_\beta Z^J w\big|
\lesssim C_1\eps t^{-1} s^{-1+\delta}.
\eea

Next if $|I|\leq N-1$ and $N$ is sufficiently large to apply \eqref{eq:m1pointwise-KG1} (i.e., $\tfrac{N-1}{2}\leq N-3$), then using \eqref{eq:m1-998} we obtain
\bea\notag
\|(s/t)Z^I\del_t \del_\alpha w \|_{L^2_f(\Hcal_{s})}
&\lesssim 
	\sum_{|J|\leq N}  \|(s/t) (t-r)^{-1}\del Z^Jw \|_{L^2_f(\Hcal_{s})} 
\\
& \quad 
+ \sum_{|J|\leq \frac{N-1}{2}} \big\| \frac{s}{(t-r)} Z^J v \big\|_{L^\infty(\Hcal_{s})} \cdot
\sum_{|J'|\leq N-1} \| Z^{J'}v \|_{L^2_f(\Hcal_{s})} 
\\
&\lesssim
	 C_1\eps s^\delta.
\eea
Similarly if $|I|\leq N-1$  then using \eqref{eq:m1-999} we find
\bea \notag
\|s(t-r)^{-1}Z^I\underdel_a \del_\alpha w \|_{L^2_f(\Hcal_{s})}
&\lesssim 
	\sum_{|J|\leq N}  \|(s/t) (t-r)^{-1}\del Z^Jw \|_{L^2_f(\Hcal_{s})} 
\lesssim
	 C_1\eps s^\delta.
\eea
Putting this all together, we have, for $|I|\leq N-1$,
\bea \notag
\sum_{\substack{| I_1 | + | I_2 | \leq | I |}} & \left\|(s/t)^2 Z^{I_1}  \del_t \del_\alpha w \cdot Z^{I_2}  \del_t \del_\beta w \right\|_{L^2_f(\Hcal_{s})}
\lesssim (C_1\eps)^2 s^{-2+2\delta},
\eea
and
\bea\notag
&\sum_{\substack{| I_1 | + | I_2 | \leq | I |}}  \left\|Z^{I_1} \underdel_a \del_\alpha w \cdot Z^{I_2} \del_t \del_\beta w \right\|_{L^2_f(\Hcal_{\tau})}
\\
&\quad \lesssim 
\sum_{\substack{| I_1 | \leq (N-1)/2\\ |I_2| \leq N-1 }} \left\|(t/s)Z^{I_1} \underdel_a \del_\alpha w \right\|_{L^\infty(\Hcal_{s})} \left\|(s/t)Z^{I_2}  \del_t \del_\beta w \right\|_{L^2_f(\Hcal_{s})}
\\
&\quad\quad+ \sum_{\substack{| I_2 | \leq (N-1)/2\\ |I_1| \leq N-1 }} \left\|s^{-1}(t-r)Z^{I_2} \del_t \del_\beta w  \right\|_{L^\infty(\Hcal_{s})} \left\| s(t-r)^{-1}Z^{I_1}\underdel_a \del_\alpha w\right\|_{L^2_f(\Hcal_{s})}
\\
&\quad\lesssim (C_1\eps)^2 s^{-2+2\delta},
\eea
and finally
\bea\notag
& \sum_{\substack{| I_1 | + | I_2 | \leq | I | }}\left\| Z^{I_1} \underdel_a \del_\alpha w \cdot Z^{I_2}  \underdel_b\del_\beta w  \right\|_{L^2_f(\Hcal_{s})}
\\
&\quad \lesssim 
	\sum_{|I_1|\leq (N-1)/2} \Big\|\frac{(t-r)}{s}Z^{I_1} \underdel_a \del_\alpha w \Big\|_{L^\infty(\Hcal_{s})}
	\sum_{|I_2|\leq N-1} \Big\| s(t-r)^{-1}Z^{I_2}\underdel_b\del_\beta w \Big\|_{L^2_f(\Hcal_{s})}
\\
&\quad \lesssim	(C_1\eps)^2 s^{-2+2\delta}.
\eea

Inserting all these estimates into \eqref{eq:M1energyde1}, together with the estimate of Lemma \ref{lem:model1-300} and restriction $\delta \leq 1/3$, we find,
\bea \notag
E_1(s, Z^I V)^{1/2}
&\lesssim \eps + 
 \int_{s_0}^s 
(C_1 \eps)^2 \tau^{-1-\delta} \di \tau \lesssim  \eps +  (C_1 \eps)^2.
\eea
The proof is done.
\end{proof}

\begin{corollary}\label{corol:m1impKGlow}
For all multi-indices of order $|I|\leq N-1$, there exists a $C>0$ such that 
\bea\notag
E_1(s, Z^I v)^{1/2} 
\leq
C \eps + C (C_1 \eps)^2.
\eea
\end{corollary}

\begin{proof}
By Proposition \ref{prop:model1-343}, the transformation formula $v = V+P^{\alpha\beta}\del_\alpha w \del_\beta w$ and an application of Minkowski's inequality, it suffices to control, for $|I| \leq N-1$,
\bea \label{eq:fi223}
E_1\big( s, Z^I \big( P^{\alpha\beta}\del_\alpha w\del_\beta w\big) \big).
\eea

By using the commutator estimates of Lemma \ref{lem:est-comm}, and assuming $N$ is sufficiently large in order to apply the refined estimates \eqref{eq:RefinedWeightedLinfty} ($\frac{N-1}{2}\leq N-5$ suffices), we have
\bea\notag
\big\|&Z^I \big( P^{\alpha\beta}\del_\alpha w\del_\beta w\big)\big\|_{L^2_f(\Hcal_{s})}
\\
&\lesssim
\sum_{\substack{|I_1|\leq (N-1)/2\\ |I_2|\leq N-1}} \|(t/s)(t-r)\del_\alpha Z^{I_1} w\|_{L^\infty(\Hcal_{s})}
\big\|(s/t)(t-r)^{-1} \del_\beta Z^{I_2} w \big\|_{L^2_f(\Hcal_{s})}
\\
&
\lesssim
(C_1\eps)^2.
\eea
The second term of \eqref{eq:fi223} requires a bit more care. Again by applying the refined estimates \eqref{eq:RefinedWeightedLinfty} ($\tfrac{N-1}{2}\leq N-5$ suffices), we find
\bea\notag
\Big\|&\underdel_a \Big(Z^I \big( P^{\alpha\beta}\del_\alpha w\del_\beta w\big)\Big)\Big\|_{L^2_f(\Hcal_{s})}
\\
&\lesssim
\sum_{\substack{|I_1|\leq (N-1)/2\\ |I_2|\leq N-1}}  \Big(
	\big\|\underdel_a (Z^{I_1} \del w)\cdot Z^{I_2} \del w\big\|_{L^2_f(\Hcal_{s})} 
	+ \big\|Z^{I_1}\del w\cdot\underdel_a ( Z^{I_2}\del w)\big\|_{L^2_f(\Hcal_{s})} \Big)
\\
&\lesssim
\sum_{\substack{|I_1|\leq (N-1)/2\\ |I_2|\leq N-1}} 
	\|(t/s) (t-r)\underdel_a ( \del Z^{I_1} w)\|_{L^\infty(\Hcal_{s})} \big\|(s/t)(t-r)^{-1}\del Z^{I_2}w\big\|_{L^2_f(\Hcal_{s})} 
\\
&\quad+
	\sum_{\substack{|I_1|\leq (N-1)/2\\ |I_2|\leq N-1}} 
	\|(t-r) s^{-1}\del Z^{I_1}w\|_{L^\infty(\Hcal_{s})} \big\|(s/t)(t-r)^{-1}L_a ( \del Z^{I_2} w)\big\|_{L^2_f(\Hcal_{s})} 
\\
&\lesssim
	(C_1\eps)^2  \Big( \|(t/s)(t-r)  t^{-1}s^{-1}\|_{L^\infty(\Hcal_{s})}  + \|(t-r)s^{-1} \cdot  s^{-1}\|_{L^\infty(\Hcal_{s})}s^\delta  \Big)
\\
&\lesssim
	(C_1 \eps)^2 .
\eea
A similar argument shows that
\bea\notag
\Big\|(s/t) \del_t \big(Z^{I} \big( P^{\alpha\beta}\del_\alpha w\del_\beta w\big)\big)\Big\|_{L^2_f(\Hcal_{s})}
\lesssim (C_1 \eps)^2.
\eea
Thus we obtain, for any $|I|\leq N-1$,
\bea \notag
E_1\big( s,Z^{I} \big( P^{\alpha\beta}\del_\alpha w\del_\beta w\big) \big)^{1/2} \lesssim (C_1 \eps)^2,
\eea
and so the conclusion follows by combining this with  Proposition  \ref{prop:model1-343}. 
\end{proof}

\subsubsection{Wave component}
\begin{lemma}\label{lem:m1waveimplow}
For all multi-indices of order $|I|\leq N-1$, there exists a $C>0$ such that  
\be 
E (s, Z^I w)^{1/2}
\leq C C_1 \eps + C (C_1 \eps)^2 s^\delta.
\ee
\end{lemma}

\begin{proof}
If $N$ is sufficiently large so that we can apply the estimates \eqref{eq:m1pointwise-KG1} (this requires $\frac{N-1}{2}\leq N-3$), then by the standard energy estimate of Proposition \ref{prop:BasicEnergyEstimate} with $m=0$, we have
\bea \notag
E(s, Z^I w)^{1/2} &\lesssim 
 \eps 
+ \int_{s_0}^s \| C_1 \eps t^{-1} \|_{L^\infty(\Hcal_{\tau})}\sum_{|I|\leq N-1} \| Z^I v\|_{L^2_f(\Hcal_{\tau})}
 \, \di \tau
\\
&\lesssim \eps
+ (C_1 \eps)^2 \int_{s_0}^s  \tau^{-1}  \, \di \tau
\lesssim \eps 
+ (C_1 \eps)^2 s^\delta.
\eea
The proof is finished.
\end{proof}

\subsection{Highest-order bootstraps}
We now look at the highest-order $L^2$ estimates, first for the Klein-Gordon component and then for the wave component. 
\begin{lemma}\label{lem:m1KGenergyimp}
There exists a $C>0$ such that  
\bea \notag
E_1 (s, Z^I v)^{1/2} & \leq C \eps + C(C_1 \eps)^2 s^{1+\delta},\quad |I|=N.
\eea
\end{lemma}

\begin{proof}
At this top order we study the original Klein-Gordon equation \eqref{eq:model-1-KG}.
Proposition \ref{prop:BasicEnergyEstimate} together with the refined decay estimates of \eqref{eq:RefinedWeightedLinfty} (and $\frac{N}{2}\leq N-5$) imply that
\bea\notag
E_1 (s, Z^I v)^{1/2}
& \lesssim
E_1 (s_0, Z^I v)^{1/2}
+ \sum_{\substack{|I_1|\leq N/2\\ |I_2|\leq N}} \sum_{\alpha, \beta} \int_{s_0}^s \big\|\del_\alpha Z^{I_1}w \cdot \del_\beta Z^{I_2} w \big\|_{L^2_f(\Hcal_{\tau})}\di \tau
\\
& \lesssim
\eps 
+ \sum_{\substack{|I_1|\leq N/2\\ |I_2|+\leq N}} \int_{s_0}^s \big\|(t-r) (t/s)(\del Z^{I_1}w)\|_{L^\infty(\Hcal_{\tau})} \big\| \frac{(s/t)}{(t-r)}( \del Z^{I_2} w )\big\|_{L^2_f(\Hcal_{\tau})}\di \tau
\\
&\lesssim
\eps + (C_1 \eps)^2 \int_{s_0}^s \tau^\delta \di \tau
\lesssim \eps + (C_1 \eps)^2 s^{1+\delta}.
\eea
The proof is complete.
\end{proof}

\begin{lemma}\label{lem:m1waveweighted}
There exists a constant $C>0$ such that  
\bea\notag
\big\| (t-r)^{-1} (s/t) \del_\alpha Z^I w \big\|_{L^2_f(\Hcal_s)}
&\leq C\eps + C(C_1 \eps)^{3/2}, \quad |I|\leq N-1,
\\
\big\| (t-r)^{-1} (s/t) \del_\alpha Z^I w \big\|_{L^2_f(\Hcal_s)}
&\leq C\eps + C(C_1 \eps)^{3/2} s^{\delta}, \quad |I|=N.
\eea
\end{lemma}

\begin{proof}
Using \eqref{eq:m1pointwise-KG1}, we see that, for $|I|\leq N-3$,
\be \notag
\big| (t-r)^{-1} Z^I v \big|
\lesssim C_1 \eps \frac{t+r}{s^2} t^{-1}
\lesssim C_1 \eps s^{-2}.
\ee
So, provided $\frac{N}{2}\leq N-3$ in order to apply the estimates \eqref{eq:m1pointwise-KG1}, Proposition \ref{prop:GhostWeight} with $m=0$ and $\gamma = 2$ gives, for $|I|=N$,  
\begin{align*}
&\big\| (t-r)^{-1} (s/t) \del_\alpha Z^I w \big\|_{L^2_f(\Hcal_s)}^2
\lesssim
 \eps^2 + \sum_{\substack{|I_1|\leq N/2\\ |I_2|\leq N}} \int_{s_0}^s \int_{\Hcal_{\tau}}\left| \frac{(\tau/t)}{(t-r)^2} (Z^{I_1} v) (Z^{I_2} v) \del_t Z^I w\right| \di x \di \tau
\\
&\lesssim
 \eps^2 + \sum_{\substack{|I_1|\leq N/2\\ |I_2|\leq N}} \int_{s_0}^s \Big\| (t-r)^{-1} (Z^{I_1} v) (Z^{I_2} v)  \Big\|_{L^2_f(\Hcal_\tau)} 
	\Big\|\frac{\tau/t}{(t-r)} \del_t Z^I w \Big\|_{L^2_f(\Hcal_{\tau})} \di \tau
\\
&\lesssim
 \eps^2 + \sum_{|I_2|\leq N} \int_{s_0}^s C_1 \eps \tau^{-2} \| Z^{I_2} v \|_{L^2_f(\Hcal_\tau)} 
	\Big\|\frac{\tau/t}{(t-r)} \del_t Z^I w \Big\|_{L^2_f(\Hcal_{\tau})} \di \tau
\\
&\lesssim
 \eps^2 + (C_1 \eps)^3 \int_{s_0}^s \tau^{-2}\tau^{1+\delta}\tau^\delta \di \tau
\lesssim
\eps^2 + (C_1 \eps)^3 s^{2\delta}.
\end{align*}

Repeating the argument for $|I|\leq N-1$ we find
\bea\notag
&\big\| (t-r)^{-1} (s/t) \del_\alpha Z^I w \big\|_{L^2_f(\Hcal_s)}^2
\\
&\lesssim
 \eps^2 
+ \sum_{\substack{|I_1|\leq \frac{N-1}{2}\\ |I_2|\leq N-1}} \int_{s_0}^s \Big\| (t-r)^{-1} (Z^{I_1} v) (Z^{I_2} v)  \Big\|_{L^2_f(\Hcal_\tau)} 
	\Big\|\frac{\tau/t}{(t-r)} \del_t Z^I w \Big\|_{L^2_f(\Hcal_{\tau})} \di \tau
\\
&\lesssim
 \eps^2 + (C_1 \eps)^3 \int_{s_0}^s \tau^{-2} \di \tau
\lesssim
\eps^2 + (C_1 \eps)^3.
\eea
The proof is done.
\end{proof}

\begin{proof}[Proof of Theorem \ref{thm1}.]
By bringing together the results of Corollary \ref{corol:m1impKGlow} and Lemmas \ref{lem:m1waveimplow}, \ref{lem:m1KGenergyimp}, \ref{lem:m1waveweighted}, we see that for a fixed $0<\delta\ll1$ there exists an $N\in\mathbb{N}$ (in fact $N \geq 10$ is imposed by Lemma \ref{lem:m1KGenergyimp}) and an $\eps_0>0$ sufficiently small that for all $0<\eps\leq \eps_0$ we have
\begin{align*}
s^{-\delta}E(s, Z^I w)^{1/2} + \big\|  (s/t) (t-r)^{-1}\del Z^I w \big\|_{L^2_f(\Hcal_s)} + E_1 (s, Z^I v)^{1/2}
&\leq \tfrac12 C_1 \eps ,
\quad
&|I| &\leq N-1,
\\
s^{-\delta} \big\|  (s/t) (t-r)^{-1}\del Z^I w \big\|_{L^2_f(\Hcal_s)} + s^{-(1+\delta)}E_1 (s, Z^I v)^{1/2}
&\leq \tfrac12 C_1 \eps ,
\quad
&|I| &= N.
\end{align*}
We can now bring together all the components of the bootstrap argument. Initial data posed on the hypersurface $\{t_0=0\}$, localised in the unit ball $\{x\in\RR^2:r\leq 1\}$ and satisfying \eqref{eq:thm1data} can be developed as a solution of \eqref{eq:model-1} to the initial hyperboloid $\Hcal_{s_0=2}$ with the smallness conserved via the bound \eqref{eq:m1BApre}. Next, by classical local existence results for quasilinear hyperbolic PDEs, the bounds \eqref{eq:BA-Easy} hold whenever the solution exists. Clearly $s^*>s_0$ and, moreover, if $s^*<+\infty$ then one of the inequalities in \eqref{eq:BA-Easy} must be an equality. However we see by choosing $C_1$ sufficiently large and $\eps_0$ sufficiently small, the bounds \eqref{eq:BA-Easy} are in fact refined. This then implies that we must have $s^*=+\infty$. 
\end{proof}
\section{Model II}\label{sec:Model2}
The second model reads 
\begin{subequations}
\begin{align}
-\Box \tw &= P^\alpha \del_\alpha (\tv^2), \label{eq:model-2-wave}
\\
-\Box \tv + \tv &= \tw^2, \label{eq:model-2-KG}
\end{align}
\end{subequations}
with initial data $\big(\tw, \del_t \tw, \tv, \del_t \tv \big) (t_0) = (\tw_0, \tw_1, \tv_0, \tv_1)$ and where $P^\alpha$ are arbitrary constants.

\paragraph{Novelties in the proof.} 
Similar to Model I, one nonlinearity is at the borderline of integrability and the other strictly below the borderline of integrability. 
We aim to use the transformation 
\be  \label{intro:m2transfV}
\tV = \tv-\tw^2
\ee 
to improve our control on $\tv$ and significantly, unlike in Model I, we can use this transformation even at the highest order energies. To control the energies of $\tv$ in terms of those of $\tV$ however we need to control the error term $\tw^2$ in $L^2$.  
One way to obtain robust $L^2$ control of an undifferentiated wave field is via a conformal energy estimate. In order to avoid uncontrollable growth in the conformal energy estimate, we use the crucial total-derivative structure in the wave nonlinearity which allows us to perform a decomposition, due to Katayama \cite{Katayama12a}, of the form
\be \label{intro:m2transfW}
\tw=\tWz+P^\alpha \del_\alpha \tW.
\ee
In particular, $\tWz$ satisfies a homogeneous wave equation with initial data $(\tw_0, \tw_1)$, and thus good conformal energy estimates, while $\tW$ satisfies an inhomogeneous wave equation
\be\notag
-\Box \tW = \tv^2
\ee
with zero initial data. By making use of appropriate $(t-r)$-weighted energy estimates, the above transformations and decompositions we can close our bootstrap argument. 

\subsection{The bootstraps and preliminary estimates}
Fix $N\in\mathbb{Z}$ a large integer. By the definition of the hyperboloidal energy functional and smallness of the data from \eqref{eq:thm2data}, there exists a constant $C_0>0$ such that on the initial hyperboloid $\Hcal_{s_0}$ the following energy bounds hold for all $|I|\leq N$:
\bea\label{eq:m2BApre}
E(s_0, Z^I w)^{1/2}  + E_1 (s_0, Z^I v)^{1/2} \leq C_0 \eps. 
\eea
Next we fix $\delta, \eps, C_1$ positive constants such that $\delta\ll 1$.  \underline{For the rest of section \ref{sec:Model2}} we assume, without restating the fact, that on a hyperbolic time interval $[s_0, s^*)$ the following bootstrap assumptions hold 
\begin{subequations} \label{eq:BA-2}
\begin{align}
E(s, Z^I \tw)^{1/2} + E_1 (s, Z^I \tv)^{1/2}  &\leq C_1 \eps s^\delta,
\qquad
&|I| &\leq N, \label{eq:BA-2-high}
\\
s (t-r)^{-2\delta} |Z^I \tw| + s |\del Z^I \tw| + t |Z^I \tv|
&\leq C_1 \eps,
\qquad
&|I| &\leq N-3, \label{eq:BA-2-low}
\end{align}
\end{subequations}
where we define
\be 
s^* \define \sup_{s\geq s_0} \{\eqref{eq:BA-2} \text{ hold on }  [s_0, s)\}.
\ee
Taking $C_1\gg C_0$ we see, by continuity of the above $L^2$ functionals, that $s^*>s_0$. 
We first assume $s^*<+\infty$ and then derive a contradiction to assert that $s^*=+\infty$.

The bootstrap assumptions \eqref{eq:BA-2} and the definition \eqref{eq:2energy} imply the following 
\bea \label{eq:BA2-implic}
\| Z^I \tv\|_{L^2_f(\Hcal_s)} + \| (s/t) \del Z^I \tv\|_{L^2_f(\Hcal_s)}
\leq &C_1 \eps s^\delta,
\qquad
|I|\leq N.
\eea
Unlike in Model I, the bootstrap assumptions \eqref{eq:BA-Easy} and Sobolev estimates of Lemma \ref{lem:sobolev} do not give sharp decay estimates for the Klein-Gordon field. Nevertheless we have
\begin{subequations}\label{eq:M2-pointwise-prelim}
\begin{align} 
|(s/t)\del Z^I \tv| +| Z^I \tv| 
\lesssim &C_1 \eps t^{-1}s^\delta,
\qquad
|I| \leq N-2.\label{eq:M2-pointwise-prelim1}
\\
|\del Z^I \tw| 
\lesssim &C_1 \eps s^{-1} s^{\delta},
\qquad
|I| \leq N-2.\label{eq:M2-pointwise-prelim3}
\end{align}
\end{subequations}
Note also that by the commutator estimate \eqref{eq:est-cmt1} and \eqref{eq:BA-2-low}, for $|I|\leq N-4$,
\bea \label{eq:M2-pointwise-prelim2}
|\underdel_a Z^I \tw| = t^{-1} |L_a Z^I \tw|
\lesssim t^{-1} \sum_{|I|\leq N-3}|Z^I \tw|
\lesssim C_1 \eps t^{-1} s^{-1} (t-r)^{2\delta}.
\eea

\subsection{Transformations}
We start by performing a transformation in the Klein-Gordon variable. 
That is, the variable $\tV\define\tv-\tw^2$ satisfies the PDE
\be\label{eq:model2-KGtransf1}
-\Box \tV+\tV=2Q_0( \tw, \tw) -2\tw P^\alpha\del_\alpha(\tv^2).
\ee
We can then perform one more transformation to $\hV \define \tV-Q_0( \tw, \tw)$. The variable $\hV$ satisfies the PDE
\bea \label{eq:model2-KGtransf2}
-\Box \hV+\hV&=\eta^{\gamma\lambda}\del_\gamma (\eta^{\alpha\beta} \del_\alpha \tw)\del_\lambda ( \del_\beta \tw) -2\tw P^\alpha\del_\alpha(\tv^2)-2\eta^{\alpha\gamma}\del_\alpha (P^\beta\del_\beta(\tv^2))\del_\gamma \tw
\\
&=2Q_0( \del_\alpha \tw, \del^\alpha \tw) -4\tw \tv P^\alpha\del_\alpha \tv-4P^\beta(\del_\alpha \tv \del_\beta \tv + \tv \del_\alpha \del_\beta \tv) \del^\alpha \tw.
\eea

Next we turn to decompositions for the wave variable. 
Following Katayama \cite{Katayama12a}, we introduce the new wave variables $\tWz$ and $\tW$ satisfying
\bea \label{eq:model2-wavetransf1}
-\Box \tWz = 0,
\qquad
\big(\tWz, \del_t \tWz) (s_0) = (\tw, \del_t\tw)(s_0),
\eea
\bea  \label{eq:model2-wavetransf2}
-\Box \tW = \tv^2,
\qquad
\big(\tW, \del_t \tW) (s_0) = (0, 0).
\eea
Our variable $\tw$ is then given by $\tw=\tWz+P^\alpha \del_\alpha \tW$. In summary, the transformations and decompositions are given by
\bel{eq:model2-tr}
\tV=\tv-\tw^2,\quad
\hV = \tV-2Q_0( \tw, \tw),\quad
\tw=\tWz+P^\alpha \del_\alpha \tW.
\ee

\subsubsection{Consequences of the transformations}
Firstly, since $Z^I \tWz$ is a solution to the homogeneous wave equation for all $|I|\leq N$, we have by Kirchoff's formula (or Lemma \ref{lem:supwave}) the decay estimate
\bea \label{eq:m2conf102}
|Z^I \tWz|
&\lesssim \eps s^{-1}, \quad |I|\leq N-2.
\eea

\begin{lemma}
For all multi-indices of order  $|I|\leq N$, we have
\bea \label{eq:m2conf103}
\big\| (s / t) Z^I \tw \big\|_{L^2_f(\Hcal_s)} 
&\lesssim C_1 \eps s^\delta.
\eea
\end{lemma}

\begin{proof}
From the conformal energy estimate of Proposition \ref{lem:ConformalEnergy} we have
\be \notag
E_{con} (s,Z^I \tWz)^{1/2}
\leq 
E_{con} (s_0, Z^I \tWz)^{1/2}
\leq C \eps.
\ee
Furthermore by estimate \eqref{eq:l2type-wave} from Proposition \ref{lem:ConformalEnergy} we have
\bea \label{eq:m2conf101}
\big\| (s / t) Z^I \tWz \big\|_{L^2_f(\Hcal_s)} 
&\lesssim 
\big\| (s / t) Z^I \tWz \big\|_{L^2_f(\Hcal_{s_0})} + \int_{s_0}^s \tau^{-1}
E_{con} (s_0, Z^I \tWz)^{1/2} \di\tau
\\
&\lesssim \eps
+ \eps \int_{s_0}^s \tau^{-1} \di \tau
\lesssim \eps s^\delta.
\eea
Next, if $N$ is sufficiently large to apply the sharp pointwise estimate \eqref{eq:BA-2-low} ($\frac{N}{2}\leq N-3$ suffices), then by the standard hyperboloidal energy estimate of Proposition \ref{prop:BasicEnergyEstimate} with $m=0$, we have
\bea \label{eq:222}
E(s, Z^I \tW)^{1/2} 
&\leq E(s_0, Z^I \tW)^{1/2} +\int_{s_0}^s \|Z^I \tv^2
\|_{L^2_f(\Hcal_{\tau})}\di \tau
\\
&\lesssim \eps
+ \sum_{\substack{|I_1|\leq N/2,\, |I_2|\leq N}}\int_{s_0}^s\|Z^{I_1}\tv\|_{L^\infty(\Hcal_{\tau})} \|Z^{I_2} \tv\|_{L^2_f(\Hcal_{\tau})}\di\tau
\\
&\lesssim \eps + \int_{s_0}^s (C_1\eps \tau^{-1}) (C_1 \eps \tau^\delta) \di\tau
\lesssim \eps + (C_1\eps)^2 s^\delta.
\eea

We thus obtain the following control on the undifferentiated wave component 
\bea \label{eq:583}
\|(s/t) Z^I \tw \|_{L^2_f(\Hcal_s)} 
&\lesssim  \|(s/t) Z^I \tWz \|_{L^2_f(\Hcal_s)} + \sum_\alpha\|(s/t) Z^I \del_\alpha \tW \|_{L^2_f(\Hcal_s)} 
\\
&\lesssim \eps s^\delta + \sum_{|J|\leq |I|, \, \alpha} \|(s/t) \del_\alpha Z^J  \tW \|_{L^2_f(\Hcal_s)} 
\lesssim C_1 \eps s^\delta.
\eea
The proof is complete.
\end{proof}

\subsection{Refined pointwise estimates}
\subsubsection{The Klein-Gordon component}
In the case $|I|\leq N-1$ we can use the transformation $\hV = \tV-Q_0( \tw,  \tw)$ with associated PDE \eqref{eq:model2-KGtransf2} to obtain uniform energy estimates for $\hV$, which we can use to recover one sharp pointwise estimate in \eqref{eq:BA-2-low}. 

\begin{lemma}\label{lemma:M2nulllest}
There exists a constant $C>0$ such that
\bea \notag
\| Z^I Q_0( \del_\alpha \tw, \del^\alpha \tw)\|_{L^2_f(\Hcal_{s})}
\leq C(C_1 \eps)^2 s^{-2+3\delta},
\qquad
|I|\leq N-1.
\eea
\end{lemma}

\begin{proof}
The proof works in a way similar to Proposition \ref{prop:model1-343} for Model I. We omit the details.

\end{proof}

\begin{proposition}\label{prop:m2unifhV}
For all multi-indices of order $|I|+|J|\leq N-1$ there exists a constant $C>0$ such that
\bea \notag
E_1(s, Z^I \hV)^{1/2} \leq C\eps + C(C_1\eps)^2.
\eea
\end{proposition}
\begin{proof}
By the energy estimate of Proposition \ref{prop:BasicEnergyEstimate}, we have
\bea\label{eq:377}
E_1(s, Z^I \hV)^{1/2}
&\leq 
E_1(s_0, Z^I \hV)^{1/2}
\\
&
+ \int_{s_0}^s \Big(
\| Z^I Q_0( \del_\alpha \tw, \del^\alpha \tw)\|_{L^2_f(\Hcal_{\tau})}
+ \| Z^I (\tw \cdot\tv \cdot\del \tv)\|_{L^2_f(\Hcal_{\tau})}
\\
&\quad\qquad + \| Z^I (\del \tv \cdot \del \tv \cdot \del \tw)\|_{L^2_f(\Hcal_{\tau})} + \| Z^I (\tv \cdot \del \del \tv\cdot \del \tw)\|_{L^2_f(\Hcal_{\tau})}\Big)
 \, \di \tau.
\eea
We treat the four terms under the integral of \eqref{eq:377} separately. The first term is covered by Lemma \ref{lemma:M2nulllest}. 
For the second term under the integral in \eqref{eq:377} we use \eqref{eq:583}, \eqref{eq:BA-2-low}, \eqref{eq:M2-pointwise-prelim3} (this requires $\tfrac{N-1}{2}\leq N-3$), to obtain
\begin{align*}
\|Z^I(\tw \tv \del\tv )\|_{L^2_f(\Hcal_{s})} 
&\lesssim
\sum_{\substack{|I_2|+|I_3|\leq (N-1)/2  \\ |I_1|\leq N}} 
\|(t/s)Z^{I_2} \tv \cdot Z^{I_3} \del\tv\|_{L^\infty(\Hcal_{s})} \|(s/t)Z^{I_1} \tw\|_{L^2_f(\Hcal_{s})} 
\\
&\quad+\sum_{\substack{|I_1|+|I_3|\leq (N-1)/2  \\ |I_2|\leq N}}
\|Z^{I_1} \tw \cdot Z^{I_3} \del\tv\|_{L^\infty(\Hcal_{s})} \| Z^{I_2} \tv \|_{L^2_f(\Hcal_{s})} 
\\
&\quad+\sum_{\substack{|I_1|+|I_2|\leq (N-1)/2 \\ |I_3|\leq N}}
\|(t/s) Z^{I_1} \tw \cdot Z^{I_2} \tv \|_{L^\infty(\Hcal_{s})} \| (s/t) Z^{I_3}  \del\tv \|_{L^2_f(\Hcal_{s})} 
\\
&\lesssim
	(C_1\eps)^3 s^{-2+4\delta}.
\end{align*}

The third and the fourth terms under the integral in \eqref{eq:377} are even easier to estimate with the same bound as the second term under the integral in \eqref{eq:377} . So we do not provide further details. 

Finally, requiring $\delta<1/4$, we find
\bea\notag
E_1(s, Z^I \hV)^{1/2}
&\lesssim 
\eps
+ \int_{s_0}^s (C_1\eps)^3 \tau^{-2+4\delta} \di \tau
\lesssim \eps + (C_1\eps)^2.
\eea

The proof is finished.
\end{proof}

\begin{corollary}\label{m2:refinedpointKG}
There exists a constant $C>0$ such that
\bea\notag
|Z^I \tv| \leq &C \big(\eps + (C_1 \eps)^2\big)t^{-1},
\qquad
|I|\leq N-3.
\eea
\end{corollary}

\begin{proof}
We first combine the Sobolev estimate of Lemma \ref{lem:sobolev} with the uniform energy estimate of Proposition \ref{prop:m2unifhV} to obtain
\bea\notag
|Z^I \hV| \lesssim \big(\eps + (C_1 \eps)^2 \big) t^{-1}, \quad |I|\leq N-3.
\eea 
Then by the transformation identity \eqref{eq:model2-tr}, commutator estimate \eqref{eq:est-cmt1}, pointwise wave estimate \eqref{eq:BA-2-low} and fact $t\leq s^2$ within $\Kcal$, we find, for $|I|\leq N-3$,
\bea\notag
|Z^I \tV| &\lesssim |Z^I \hV|+|Z^I Q_0(\tw, \tw)|
\lesssim \big(\eps + (C_1 \eps)^2\big) t^{-1}+|Z^I \del \tw|^2
\lesssim \big(\eps + (C_1 \eps)^2 \big)t^{-1}.
\eea
Finally, by the transformation identity \eqref{eq:model2-tr}, decay estimate \eqref{eq:BA-2-low} and fact that $\delta \ll 1$, we have on $\Hcal_s$, for $|I|\leq N-3$,
\bea\notag
|Z^I \tv| &\lesssim |Z^I \tV|+|Z^I (\tw)^2|
\lesssim \big(\eps + (C_1 \eps)^2\big) t^{-1}.
\eea
The proof is complete.
\end{proof}

\subsubsection{The wave component}
\begin{lemma}
There exists a constant $C>0$ such that
\begin{align}
&\big\| (s/t) (t-r)^{-2\delta} \del_\alpha Z^I \tW \big\|_{L^2_f(\Hcal_s)}\leq C\eps + C(C_1\eps)^2, \quad &|I|&\leq N,
\\
&|\del Z^I \tW| \leq C (\eps + (C_1\eps)^2) s^{-1}(t-r)^{2\delta}, \quad &|I|&\leq N-2. \label{eq:m2-point-weight-W}
\end{align}
\end{lemma}

\begin{proof}
We derive the weighted energy estimate for $\tW$ using the equation \eqref{eq:model2-wavetransf2} and Proposition \ref{prop:GhostWeight} with $m=0$ and $\gamma=4\delta$.  If $N$ is sufficiently large to apply the sharp pointwise estimate \eqref{eq:BA-2-low} ($\frac{N}{2}\leq N-3$ suffices), then we have, for $|I|+|J|\leq N$,
\bea\notag
\quad\big\| &(t-r)^{-2\delta} (s/t) \del_\alpha Z^I \tW \big\|_{L^2_f(\Hcal_s)}^2
\\
&\lesssim \big\| (t-r)^{-2\delta} (s/t) \del_\alpha Z^I \tW \big\|_{L^2_f(\Hcal_{s_0})}^2 
\\
	&\quad + \sum_{\substack{|I_1|\leq N/2,\, |I_2|\leq N}} 
	\int_{s_0}^s \Big\|\frac{(\tau/t)}{(t-r)^{4\delta}} (Z^{I_1} \tv) (Z^{I_2} \tv) \del_t Z^I \tW\Big\|_{L^1_f(\Hcal_\tau)} \di \tau 
\\
& \lesssim
\big\| (s/t) \del_\alpha Z^I \tW \big\|_{L^2_f(\Hcal_{s_0})}^2  
\\
&\quad + \sum_{\substack{|I_1|\leq N/2,\, |I_2|\leq N}} \int_{s_0}^s \Big\|\frac{\tau^\delta}{(t-r)^{2\delta}} Z^{I_1} v\Big\|_{L^\infty(\Hcal_{\tau})} \Big\|\frac{(\tau/t)}{(t-r)^{2\delta}} \del_t Z^I \tW \Big\|_{L^2_f(\Hcal_\tau)} \big\|\tau^{-\delta}Z^{I_2} \tv\big\|_{L^2_f(\Hcal_\tau)} \di \tau 
\\
& \lesssim
\eps^2 + (C_1 \eps)^2 \int_{s_0}^s \tau^{-1-\delta}  \Big\|\frac{(\tau/t)}{(t-r)^{2\delta}} \del_t Z^I \tW \Big\|_{L^2_f(\Hcal_\tau)} \di \tau.
\eea
Taking the sum over all $|I|\leq N$ and applying Gr\"onwall's inequality we find
\bea\notag
\sum_{|I|+|J|\leq N}\big\| &(t-r)^{-2\delta} (s/t) \del_\alpha Z^I \tW \big\|_{L^2_f(\Hcal_s)}^2
\lesssim \eps^2
+\Big(\int_{s_0}^s (C_1\eps)^2\tau^{-1-\delta} \di \tau \Big)^2
\eea
and thus, for any $|I|\leq N$,
\bea\notag
\big\| &(s/t) (t-r)^{-2\delta} \del_\alpha Z^I \tW \big\|_{L^2_f(\Hcal_s)}\lesssim \eps + (C_1\eps)^2.
\eea
Applying the weighted Sobolev estimate of \eqref{eq:weightedSobolev} we can find, for $|I|\leq N-2$,
\bea \notag
|\del Z^I \tW| \lesssim (\eps + (C_1\eps)^2) s^{-1}(t-r)^{2\delta}.
\eea
The proof is done.
\end{proof}

\begin{corollary}\label{corol:m2imp-w-point}
For any multi-indices of order  $|I|\leq N-3$, there exists a constant $C>0$ such that
\bea\notag
|Z^I \tw| \leq C(\eps + (C_1\eps)^2) s^{-1} (t-r)^{2\delta}.
\eea
\end{corollary}
\begin{proof}
The proof follows by combining the decay estimates \eqref{eq:m2-point-weight-W} and \eqref{eq:m2conf102} with the decomposition formula \eqref{eq:model2-tr} and the commutator estimate \eqref{eq:est-cmt1}.
\end{proof}

\subsection{Refined highest-order energy estimates}
\subsubsection{Wave component}
\begin{lemma}\label{lem:m2impwave}
For all multi-indices of order $|I|\leq N$ there exists a constant $C>0$ such that
\bea\notag
E(s,Z^I \tw)^{1/2}
\leq C\eps + C(C_1 \eps)^2 s^\delta.
\eea
\end{lemma}

\begin{proof}
For $|I|\leq N$, and $N$ sufficiently large to apply the sharp pointwise estimate \eqref{eq:BA-2-low} (this requires $N/2\leq N-4$), we have
\bea\notag
E(s,Z^I \tw)^{1/2}
&\leq 
E(s_0, Z^I \tw)^{1/2}
+
\int_{s_0}^s \| Z^I (P^\alpha \del_\alpha (\tv^2))\|_{L^2_f(\Hcal_{\tau})} \, \di\tau
\lesssim \eps + (C_1 \eps)^2 s^\delta.
\eea
The proof is complete.
\end{proof}

\begin{corollary}\label{m2:refinedpointdelwave}
There exists a constant $C>0$ such that
\bea\notag
|\del Z^I \tw| \leq &\big(\eps + (C_1\eps)^2 \big)s^{-1},
\qquad
|I|\leq N-3.
\eea
\end{corollary}

\begin{proof}
We use Lemmas \ref{lem:est-comm} and \ref{lem:Wavedeldel} together with the pointwise estimates \eqref{eq:BA-2-low}, \eqref{eq:m2conf102} and \eqref{eq:m2-point-weight-W}, and the decomposition formula \eqref{eq:model2-tr} with associated PDE \eqref{eq:model2-wavetransf2}. We find, for $|I|\leq N-3$,
\bea\notag
|Z^I \del \tw|
&\lesssim |Z^I \del \tWz| + |Z^I \del( P^\alpha \del_\alpha\tW)|
\\
&\lesssim \sum_{|I'|\leq |I|}|\del Z^{I'} \tWz|
+ \frac{1}{(t-r)}
	\sum_{|J|\leq 1} \big| \del Z^{J} (Z^I \tW) \big| 
 + \frac{t}{(t-r)} \Big(\sum_{\substack{|I'|\leq |I|}} \big|Z^{I'} \tv \big|\Big)^2
\\
&\lesssim \big(\eps + (C_1\eps)^2 \big)s^{-1}.
\eea
The proof is done.
\end{proof}

\subsubsection{Klein-Gordon component}
\begin{lemma}\label{lem:m2impKG}
There exists a constant $C>0$ such that
\bea\notag
E_1(s, Z^I \tv)^{1/2}
\leq C \eps + (C_1 \eps)^2 s^\delta,
\qquad
|I|+|J|\leq N.
\eea
\end{lemma}

\begin{proof}
By applying the energy estimate from Proposition \ref{prop:BasicEnergyEstimate} to the equation \eqref{eq:model2-KGtransf1}, we find that, for any multi-indices of order $|I|+|J|\leq N$,
\bea \label{eq:641}
E_1(s, Z^I \tV)^{1/2}
&\leq
E_1(s_0, Z^I \tV)^{1/2}
+ \int_{s_0}^s  \|Z^I Q_0(\tw, \tw)\|_{L^2_f(\Hcal_{\tau})} +\|Z^I(\tw \tv \del\tv )\|_{L^2_f(\Hcal_{\tau})} \, \di \tau.
\eea
We estimate the above two terms under the integral separately. 

For the first term, we can apply Lemma \ref{lem:null}  to obtain
\bea\label{eq:M2nullstuff} 
|  Z^I Q_0(\tw, \tw) |
& \lesssim
(s/t)^2 \sum_{\substack{| I_1 | + | I_2 | \leq | I | }} \big| Z^{I_1}  \del_t \tw \cdot Z^{I_2}  \del_t \tw \big|
\\
&\quad
	+ \sum_{\substack{| I_1 | + | I_2 | \leq | I | }}\sum_{a, \beta} \big| Z^{I_1}  \underdel_a \tw \cdot Z^{I_2}  \underdel_\beta\tw \big| .
\eea
We start by looking at the first term of \eqref{eq:M2nullstuff}. By symmetry, if $|I_1|\leq N/2\leq N-3$  we can use Lemma \ref{lem:est-comm} and the sharp estimates \eqref{eq:BA-2} to find
\bea \notag
\|(s/t)^2 Z^{I_1}  \del_t \tw \cdot Z^{I_2}  \del_t \tw\|_{L^2_f(\Hcal_{s})}
\lesssim \|(s/t)(C_1\eps s^{-1})\|_{L^\infty(\Hcal_{s})} (C_1 \eps s^\delta)\lesssim (C_1\eps)^2 s^{-1+\delta}.
\eea
Next we look at the second term of \eqref{eq:M2nullstuff}. If $|I_1|\leq N/2\leq N-4$  we have, using the commutator estimates of Lemma \ref{lem:est-comm}, and the estimate \eqref{eq:M2-pointwise-prelim2},
\bea\notag
\|(t/s) Z^{I_1}  \underdel_a \tw \cdot (s/t) Z^{I_2}  \underdel_\beta\tw \|_{L^2_f(\Hcal_{s})}
&\lesssim \|(t/s)(C_1\eps t^{-1}s^{-1}(t-r)^{2\delta})\|_{L^\infty(\Hcal_{s})} (C_1 \eps s^\delta)
\\
&\lesssim (C_1\eps)^2 s^{-1+\delta}.
\eea
While if $|I_2|\leq N/2\leq N-4$  we have, using the commutator estimates of Lemma \ref{lem:est-comm}, \eqref{eq:BA-2} and \eqref{eq:M2-pointwise-prelim2},
\bea\notag
\|Z^{I_1}  \underdel_a \tw \cdot Z^{I_2}  \underdel_\beta\tw \|_{L^2_f(\Hcal_{s})}
\lesssim \|(C_1\eps t^{-1}s^{-1}(t-r)^{2\delta}) \|_{L^\infty(\Hcal_{s})}(C_1 \eps s^\delta)
\lesssim (C_1\eps)^2 s^{-1+\delta}.
\eea

For the second term under the integral in \eqref{eq:641}, which is cubic, we easily get
$$
\|Z^I(\tw \tv \del\tv )\|_{L^2_f(\Hcal_{s})} 
\lesssim
(C_1\eps)^3 s^{-2+4\delta}.
$$


Inserting all these estimates into \eqref{eq:641} we find
\bea\notag
E_1(s, Z^I \tV)^{1/2}
&\lesssim
\eps
+ \int_{s_0}^s (C_1\eps)^2 \tau^{-1+\delta}\, \di \tau
\lesssim \eps + (C_1 \eps)^2 s^\delta.
\eea

Finally we can convert this information back into an estimate on the original variable $\tv$. By Minkowski's inequality, this gives
\bea \label{eq:807}
E_1(s, Z^I \tv)^{1/2} &\lesssim
E_1(s, Z^I \tV)^{1/2} + E_1(s, Z^I \tw^2)^{1/2}
\\
&\lesssim \eps + (C_1 \eps)^2 s^\delta + 
\|(s/t)\del_t Z^I( \tw^2)\|_{L^2_f(\Hcal_{s})} +
\sum_{a}\|\underdel_a Z^I( \tw^2)\|_{L^2_f(\Hcal_{s})} \\
&\quad +
\|Z^I( \tw^2)\|_{L^2_f(\Hcal_{s})} .
\eea

We start with the third term on the right-hand-side of \eqref{eq:807}. Using \eqref{eq:m2conf103}, \eqref{eq:BA-2-low} and \eqref{eq:M2-pointwise-prelim3} (this requires $N/2\leq N-3$) we have
\bea\notag
\|(s/t)\del_t Z^I( \tw^2)\|_{L^2_f(\Hcal_{s})} 
&\lesssim
\sum_{\substack{|I_1|\leq N, \,|I_2|\leq N/2}}
\|(s/t)\del_t \big(Z^{I_1} \tw\cdot Z^{I_2} \tw\big)\|_{L^2_f(\Hcal_{s})} 
\\
&\lesssim
\sum_{\substack{|I_1|\leq N,\,|I_2|\leq N/2}}\Big(
\|Z^{I_2} \tw\|_{L^\infty(\Hcal_{s})} \|(s/t)\del_t Z^{I_1} \tw \|_{L^2_f(\Hcal_{s})} 
\\
&\qquad \qquad+
\|\del_t Z^{I_2} \tw\|_{L^\infty(\Hcal_{s})} \|(s/t)Z^{I_1} \tw \|_{L^2_f(\Hcal_{s})} \Big)
\\
&\lesssim
(C_1\eps)^2 s^\delta.
\eea

We next look at the fourth term on the right-hand-side of \eqref{eq:807}. Using \eqref{eq:m2conf103}, \eqref{eq:BA-2-low} and \eqref{eq:M2-pointwise-prelim2} (this requires $N/2\leq N-3$) we have
\bea\notag
\sum_a\|\underdel_a Z^I( \tw^2)\|_{L^2_f(\Hcal_{s})} 
&\lesssim
\sum_{\substack{|I_1|\leq N,\,|I_2|\leq N/2}}
\|\underdel_a \big(Z^{I_1} \tw\cdot Z^{I_2} \tw\big)\|_{L^2_f(\Hcal_{s})} 
\\
&\lesssim
\sum_{\substack{|I_1|\leq N,\,|I_2|\leq N/2}}\Big(
\|Z^{I_2} \tw\|_{L^\infty(\Hcal_{s})} \|\underdel_a  Z^{I_1} \tw \|_{L^2_f(\Hcal_{s})} 
\\
&\qquad \qquad+
\|(t/s)\underdel_a  Z^{I_2} \tw\|_{L^\infty(\Hcal_{s})} \|(s/t)Z^{I_1} \tw \|_{L^2_f(\Hcal_{s})} \Big)
\\
&\lesssim
(C_1\eps)^2 s^\delta.
\eea

We can now turn to the final term on the right-hand-side of \eqref{eq:807}. Using the estimate \eqref{eq:583} and \eqref{eq:BA-2-low} we find, provided $2\delta<1$,
\bea\label{eq:m2l2contr}
\|Z^I( \tw^2)\|_{L^2_f(\Hcal_{s})} 
\lesssim
\sum_{\substack{|I_1|\leq N,\,|I_2|\leq N/2}}
 \| (s/t)Z^{I_1}\tw\|_{L^2_f(\Hcal_{s})}  \|(t/s)Z^{I_2}  \tw\|_{L^\infty(\Hcal_{s})}
\lesssim 
(C_1 \eps)^2 s^\delta.
\eea

In conclusion, we have
\bea \notag
E_1(s, Z^I \tv)^{1/2} &
\lesssim \eps + (C_1 \eps)^2 s^\delta .
\eea
The proof is done.
\end{proof}

\begin{proposition}\label{prop:m2finalising}
For a fixed $0<\delta\ll1$ there exists an $N\in\mathbb{N}$ and an $\eps_0>0$ and $C_1>C_0>0$ such that for all $0<\eps\leq \eps_0$, if the bounds \eqref{eq:BA-2} hold on a hyperbolic time interval $[s_0, s_1]$, then the refined estimates
\begin{subequations} \label{eq:BA-2i}
\begin{align}
E(s, Z^I \tw)^{1/2} + E_1 (s, Z^I \tv)^{1/2}  &\leq {1\over 2} C_1 \eps s^\delta,
\qquad
&|I| &\leq N, \label{eq:BA-2-highR}
\\
s (t-r)^{-2\delta} |Z^I \tw| + s |\del Z^I \tw| + t |Z^I \tv|
&\leq {1\over 2} C_1 \eps,
\qquad
&|I| &\leq N-3, \label{eq:BA-2-lowR}
\end{align}
\end{subequations}
hold on the same hyperbolic time interval. 
\end{proposition}
\begin{proof}
The proof follows immediately by bringing together the results of corollaries \ref{corol:m2imp-w-point}, \ref{m2:refinedpointKG}, \ref{m2:refinedpointdelwave} and Lemmas \ref{lem:m2impwave} and \ref{lem:m2impKG}. The restriction of $N\geq 8$ comes from Lemma \ref{lem:m2impKG}. 
\end{proof}

\begin{proof}[Proof of Theorem \ref{thm2}.]
We can now complete the bootstrap argument. Firstly, as shown in \cite[\textsection11]{PLF-YM-book}, initial data posed on the hypersurface $\{t_0=0\}$ and localised in the unit ball $\{x\in\RR^2:r\leq 1\}$ can be developed as a solution of the PDE to the initial hyperboloid $\Hcal_{s_0=2}$ with the smallness conserved. This justifies the bound \eqref{eq:m2BApre}. Next, by classical local existence results for quasilinear hyperbolic PDEs, the bounds \eqref{eq:BA-2} hold whenever the solution exists. Clearly $s^*>s_0$ and if $s^*<+\infty$ then one of the inequalities in \eqref{eq:BA-2} must be an equality. We see then by Proposition \ref{prop:m2finalising}, that by choosing $C_1, C_0$ sufficiently large and $\eps_0$ sufficiently small, the bounds \eqref{eq:BA-2} are in fact refined. This then implies that $s^*=+\infty$ and so the local solution extends to a global one. 
\end{proof}

\end{document}